\documentclass[a4paper, 11pt, final]{amsart}    
\usepackage{latexsym, amsmath, amsthm, amssymb, setspace, verbatim, bbm}
\usepackage[all]{xy}
\usepackage[utf8]{inputenc} \usepackage[T1]{fontenc} \usepackage{lmodern}
\usepackage{hyperref, aliascnt}
\usepackage{enumitem}
\usepackage{color}
\usepackage[notcite,notref]{showkeys}
\usepackage{mathtools}

\title[Stable uniqueness for $KK^G$]{The stable uniqueness theorem for equivariant Kasparov theory}
\author{James Gabe}
\author{Gábor Szabó}
\address{Department of Mathematics and Computer Science, University of Southern\linebreak
\phantom{-}\hspace{2mm} Denmark, Campusvej 55, 5230 Odense, Denmark}
\email{gabe@imada.sdu.dk}
\address{Department of Mathematics, KU Leuven, Celestijnenlaan 200b, box 2400\linebreak
\phantom{-}\hspace{2mm} 3001 Leuven, Belgium}
\email{gabor.szabo@kuleuven.be}


\numberwithin{equation}{section}

\begin{document}

\renewcommand\matrix[1]{\left(\begin{array}{*{10}{c}} #1 \end{array}\right)}  
\newcommand\set[1]{\left\{#1\right\}}  

\newcommand{\IA}[0]{\mathbb{A}} \newcommand{\IB}[0]{\mathbb{B}}
\newcommand{\IC}[0]{\mathbb{C}} \newcommand{\ID}[0]{\mathbb{D}}
\newcommand{\IE}[0]{\mathbb{E}} \newcommand{\IF}[0]{\mathbb{F}}
\newcommand{\IG}[0]{\mathbb{G}} \newcommand{\IH}[0]{\mathbb{H}}
\newcommand{\II}[0]{\mathbb{I}} \renewcommand{\IJ}[0]{\mathbb{J}}
\newcommand{\IK}[0]{\mathbb{K}} \newcommand{\IL}[0]{\mathbb{L}}
\newcommand{\IM}[0]{\mathbb{M}} \newcommand{\IN}[0]{\mathbb{N}}
\newcommand{\IO}[0]{\mathbb{O}} \newcommand{\IP}[0]{\mathbb{P}}
\newcommand{\IQ}[0]{\mathbb{Q}} \newcommand{\IR}[0]{\mathbb{R}}
\newcommand{\IS}[0]{\mathbb{S}} \newcommand{\IT}[0]{\mathbb{T}}
\newcommand{\IU}[0]{\mathbb{U}} \newcommand{\IV}[0]{\mathbb{V}}
\newcommand{\IW}[0]{\mathbb{W}} \newcommand{\IX}[0]{\mathbb{X}}
\newcommand{\IY}[0]{\mathbb{Y}} \newcommand{\IZ}[0]{\mathbb{Z}}

\newcommand{\Ia}[0]{\mathbbmss{a}} \newcommand{\Ib}[0]{\mathbbmss{b}}
\newcommand{\Ic}[0]{\mathbbmss{c}} \newcommand{\Id}[0]{\mathbbmss{d}}
\newcommand{\Ie}[0]{\mathbbmss{e}} \newcommand{\If}[0]{\mathbbmss{f}}
\newcommand{\Ig}[0]{\mathbbmss{g}} \newcommand{\Ih}[0]{\mathbbmss{h}}
\newcommand{\Ii}[0]{\mathbbmss{i}} \newcommand{\Ij}[0]{\mathbbmss{j}}
\newcommand{\Ik}[0]{\mathbbmss{k}} \newcommand{\Il}[0]{\mathbbmss{l}}
\renewcommand{\Im}[0]{\mathbbmss{m}} \newcommand{\In}[0]{\mathbbmss{n}}
\newcommand{\Io}[0]{\mathbbmss{o}} \newcommand{\Ip}[0]{\mathbbmss{p}}
\newcommand{\Iq}[0]{\mathbbmss{q}} \newcommand{\Ir}[0]{\mathbbmss{r}}
\newcommand{\Is}[0]{\mathbbmss{s}} \newcommand{\It}[0]{\mathbbmss{t}}
\newcommand{\Iu}[0]{\mathbbmss{u}} \newcommand{\Iv}[0]{\mathbbmss{v}}
\newcommand{\Iw}[0]{\mathbbmss{w}} \newcommand{\Ix}[0]{\mathbbmss{x}}
\newcommand{\Iy}[0]{\mathbbmss{y}} \newcommand{\Iz}[0]{\mathbbmss{z}}

\newcommand{\CA}[0]{\mathcal{A}} \newcommand{\CB}[0]{\mathcal{B}}
\newcommand{\CC}[0]{\mathcal{C}} \newcommand{\CD}[0]{\mathcal{D}}
\newcommand{\CE}[0]{\mathcal{E}} \newcommand{\CF}[0]{\mathcal{F}}
\newcommand{\CG}[0]{\mathcal{G}} \newcommand{\CH}[0]{\mathcal{H}}
\newcommand{\CI}[0]{\mathcal{I}} \newcommand{\CJ}[0]{\mathcal{J}}
\newcommand{\CK}[0]{\mathcal{K}} \newcommand{\CL}[0]{\mathcal{L}}
\newcommand{\CM}[0]{\mathcal{M}} \newcommand{\CN}[0]{\mathcal{N}}
\newcommand{\CO}[0]{\mathcal{O}} \newcommand{\CP}[0]{\mathcal{P}}
\newcommand{\CQ}[0]{\mathcal{Q}} \newcommand{\CR}[0]{\mathcal{R}}
\newcommand{\CS}[0]{\mathcal{S}} \newcommand{\CT}[0]{\mathcal{T}}
\newcommand{\CU}[0]{\mathcal{U}} \newcommand{\CV}[0]{\mathcal{V}}
\newcommand{\CW}[0]{\mathcal{W}} \newcommand{\CX}[0]{\mathcal{X}}
\newcommand{\CY}[0]{\mathcal{Y}} \newcommand{\CZ}[0]{\mathcal{Z}}

\newcommand{\FA}[0]{\mathfrak{A}} \newcommand{\FB}[0]{\mathfrak{B}}
\newcommand{\FC}[0]{\mathfrak{C}} \newcommand{\FD}[0]{\mathfrak{D}}
\newcommand{\FE}[0]{\mathfrak{E}} \newcommand{\FF}[0]{\mathfrak{F}}
\newcommand{\FG}[0]{\mathfrak{G}} \newcommand{\FH}[0]{\mathfrak{H}}
\newcommand{\FI}[0]{\mathfrak{I}} \newcommand{\FJ}[0]{\mathfrak{J}}
\newcommand{\FK}[0]{\mathfrak{K}} \newcommand{\FL}[0]{\mathfrak{L}}
\newcommand{\FM}[0]{\mathfrak{M}} \newcommand{\FN}[0]{\mathfrak{N}}
\newcommand{\FO}[0]{\mathfrak{O}} \newcommand{\FP}[0]{\mathfrak{P}}
\newcommand{\FQ}[0]{\mathfrak{Q}} \newcommand{\FR}[0]{\mathfrak{R}}
\newcommand{\FS}[0]{\mathfrak{S}} \newcommand{\FT}[0]{\mathfrak{T}}
\newcommand{\FU}[0]{\mathfrak{U}} \newcommand{\FV}[0]{\mathfrak{V}}
\newcommand{\FW}[0]{\mathfrak{W}} \newcommand{\FX}[0]{\mathfrak{X}}
\newcommand{\FY}[0]{\mathfrak{Y}} \newcommand{\FZ}[0]{\mathfrak{Z}}

\newcommand{\Fa}[0]{\mathfrak{a}} \newcommand{\Fb}[0]{\mathfrak{b}}
\newcommand{\Fc}[0]{\mathfrak{c}} \newcommand{\Fd}[0]{\mathfrak{d}}
\newcommand{\Fe}[0]{\mathfrak{e}} \newcommand{\Ff}[0]{\mathfrak{f}}
\newcommand{\Fg}[0]{\mathfrak{g}} \newcommand{\Fh}[0]{\mathfrak{h}}
\newcommand{\Fi}[0]{\mathfrak{i}} \newcommand{\Fj}[0]{\mathfrak{j}}
\newcommand{\Fk}[0]{\mathfrak{k}} \newcommand{\Fl}[0]{\mathfrak{l}}
\newcommand{\Fm}[0]{\mathfrak{m}} \newcommand{\Fn}[0]{\mathfrak{n}}
\newcommand{\Fo}[0]{\mathfrak{o}} \newcommand{\Fp}[0]{\mathfrak{p}}
\newcommand{\Fq}[0]{\mathfrak{q}} \newcommand{\Fr}[0]{\mathfrak{r}}
\newcommand{\Fs}[0]{\mathfrak{s}} \newcommand{\Ft}[0]{\mathfrak{t}}
\newcommand{\Fu}[0]{\mathfrak{u}} \newcommand{\Fv}[0]{\mathfrak{v}}
\newcommand{\Fw}[0]{\mathfrak{w}} \newcommand{\Fx}[0]{\mathfrak{x}}
\newcommand{\Fy}[0]{\mathfrak{y}} \newcommand{\Fz}[0]{\mathfrak{z}}

\newcommand{\KKh}{KH}
\newcommand{\IPh}{\mathbb{P}}
\newcommand{\IDh}{\mathbb{D}}
\newcommand{\Kh}{\mathrm{h}}
\newcommand{\KKp}{HP}

\renewcommand{\phi}[0]{\varphi}
\newcommand{\eps}[0]{\varepsilon}

\newcommand{\id}[0]{\operatorname{id}}		
\renewcommand{\sp}[0]{\operatorname{Sp}}		
\newcommand{\eins}[0]{\mathbf{1}}			
\newcommand{\diag}[0]{\operatorname{diag}}
\newcommand{\ad}[0]{\operatorname{Ad}}
\newcommand{\ev}[0]{\operatorname{ev}}
\newcommand{\fin}[0]{{\subset\!\!\!\subset}}
\newcommand{\Aut}[0]{\operatorname{Aut}}
\newcommand{\dimrok}[0]{\dim_{\mathrm{Rok}}}
\newcommand{\dst}[0]{\displaystyle}
\newcommand{\cstar}[0]{\ensuremath{\mathrm{C}^*}}
\newcommand{\dist}[0]{\operatorname{dist}}
\newcommand{\ann}[0]{\operatorname{Ann}}
\newcommand{\cc}[0]{\simeq_{\mathrm{cc}}}
\newcommand{\scc}[0]{\simeq_{\mathrm{scc}}}
\newcommand{\vscc}[0]{\simeq_{\mathrm{vscc}}}
\newcommand{\scd}[0]{\preceq_{\mathrm{scd}}}
\newcommand{\ue}[0]{{~\approx_{\mathrm{u}}}~}
\newcommand{\cel}[0]{\ensuremath{\mathrm{cel}}}
\newcommand{\acel}[0]{\ensuremath{\mathrm{acel}}}
\newcommand{\sacel}[0]{\ensuremath{\mathrm{sacel}}}
\newcommand{\prim}[0]{\ensuremath{\mathrm{Prim}}}
\newcommand{\co}[0]{\ensuremath{\mathrm{co}}}
\newcommand{\GL}[0]{\operatorname{GL}}
\newcommand{\Bott}[0]{\ensuremath{\mathrm{Bott}}}
\newcommand{\tK}[0]{\ensuremath{\underline{K}}}
\newcommand{\Hom}[0]{\operatorname{Hom}}
\newcommand{\sep}[0]{\ensuremath{\mathrm{sep}}}
\newcommand{\nue}[0]{\approx_{\mathrm{n.u}}}
\newcommand{\wc}[0]{\preccurlyeq}
\renewcommand{\asymp}[0]{\sim_{\mathrm{asymp}}}

\newcommand{\greater}[0]{>}
\renewcommand{\smaller}[0]{<}
\newcommand{\vslash}[0]{|}
\newcommand{\msout}[1]{\text{\sout{$#1$}}}

\newtheorem{satz}{Satz}[section]		

\newaliascnt{corCT}{satz}
\newtheorem{cor}[corCT]{Corollary}
\aliascntresetthe{corCT}
\providecommand*{\corCTautorefname}{Corollary}
\newaliascnt{lemmaCT}{satz}
\newtheorem{lemma}[lemmaCT]{Lemma}
\aliascntresetthe{lemmaCT}
\providecommand*{\lemmaCTautorefname}{Lemma}
\newaliascnt{propCT}{satz}
\newtheorem{prop}[propCT]{Proposition}
\aliascntresetthe{propCT}
\providecommand*{\propCTautorefname}{Proposition}
\newaliascnt{theoremCT}{satz}
\newtheorem{theorem}[theoremCT]{Theorem}
\aliascntresetthe{theoremCT}
\providecommand*{\theoremCTautorefname}{Theorem}
\newtheorem*{theoreme}{Theorem}

\theoremstyle{definition}

\newaliascnt{conjectureCT}{satz}
\newtheorem{conjecture}[conjectureCT]{Conjecture}
\aliascntresetthe{conjectureCT}
\providecommand*{\conjectureCTautorefname}{Conjecture}
\newaliascnt{defiCT}{satz}
\newtheorem{defi}[defiCT]{Definition}
\aliascntresetthe{defiCT}
\providecommand*{\defiCTautorefname}{Definition}
\newtheorem*{defie}{Definition}
\newaliascnt{notaCT}{satz}
\newtheorem{nota}[notaCT]{Notation}
\aliascntresetthe{notaCT}
\providecommand*{\notaCTautorefname}{Notation}
\newtheorem*{notae}{Notation}
\newaliascnt{remCT}{satz}
\newtheorem{rem}[remCT]{Remark}
\aliascntresetthe{remCT}
\providecommand*{\remCTautorefname}{Remark}
\newtheorem*{reme}{Remark}
\newaliascnt{exampleCT}{satz}
\newtheorem{example}[exampleCT]{Example}
\aliascntresetthe{exampleCT}
\providecommand*{\exampleCTautorefname}{Example}
\newaliascnt{questionCT}{satz}
\newtheorem{question}[questionCT]{Question}
\aliascntresetthe{questionCT}
\providecommand*{\questionCTautorefname}{Question}
\newtheorem*{questione}{Question}

\newcounter{theoremintro}
\renewcommand*{\thetheoremintro}{\Alph{theoremintro}}
\newaliascnt{theoremiCT}{theoremintro}
\newtheorem{theoremi}[theoremiCT]{Theorem}
\aliascntresetthe{theoremiCT}
\providecommand*{\theoremiCTautorefname}{Theorem}
\newaliascnt{defiiCT}{theoremintro}
\newtheorem{defii}[defiiCT]{Definition}
\aliascntresetthe{defiiCT}
\providecommand*{\defiiCTautorefname}{Definition}
\newaliascnt{coriCT}{theoremintro}
\newtheorem{cori}[coriCT]{Corollary}
\aliascntresetthe{coriCT}
\providecommand*{\coriCTautorefname}{Corollary}


\begin{abstract} 
This paper examines and strengthens the Cuntz--Thomsen picture of equivariant Kasparov theory for arbitrary second-countable locally compact groups, in which elements are given by certain pairs of cocycle representations between \cstar-dynamical systems.
The main result is a stable uniqueness theorem that generalizes a fundamental characterization of ordinary $KK$-theory by Lin and Dadarlat--Eilers.
Along the way, we prove an equivariant Cuntz--Thomsen picture analog of the fact that the equivalence relation of homotopy agrees with the (a priori stronger) equivalence relation of stable operator homotopy.
The results proved in this paper will be employed as the technical centerpiece in forthcoming work of the authors to classify certain amenable group actions on Kirchberg algebras by equivariant Kasparov theory.
\end{abstract}

\maketitle

\setcounter{tocdepth}{2}

\tableofcontents


\section*{Introduction}

The operator algebraic perspective on $K$-theory has been responsible for a long-standing and fruitful exchange of ideas and results between the areas of topology and \cstar-algebras.
The extension of topological $K$-theory from spaces to \cstar-algebras has not only opened up a fresh perspective on the subject, but led to numerous new insights in the original domain of origin, many of which are so well-known that we won't attempt to survey them.
In what may arguably be considered the most powerful unifying $K$-theoretical theory enabled by the noncommutative framework, Kasparov's $KK$-theory \cite{Kasparov88} is a generalized homotopy theory for $C^\ast$-algebras connecting $K$-theory, $K$-homology, and Brown-Douglas-Fillmore theory \cite{BrownDouglasFillmore77}.
The original approach by Kasparov, commonly referred to as the Kasparov picture, is still the predominant way of treating this subject.
While this is entirely appropriate for the original historical motivation and applications, subsequent uses of $KK$-theory sometimes call for a different (though equivalent) approach that is better suited for certain other applications.

The description of $KK$-theory due to Cuntz \cite{Cuntz83}, commonly referred to as the Cuntz picture, has had many applications such as in cyclic cohomology \cite{ConnesCuntz88} or Higson's celebrated description of $KK$ as a universal functor \cite{Higson87}.
In said picture, elements of Kasparov's group $KK(A,B)$ for separable \cstar-algebras $A$ and $B$ can be described as homotopy classes of so-called Cuntz pairs, i.e., pairs of $\ast$-homomorphisms $\phi, \psi \colon A \to \CM(B \otimes \CK)$ such that $\phi(a)-\psi(a)\in B\otimes \CK$ for all $a\in A$.
A groundbreaking theorem in $KK$-theory is the so-called stable uniqueness theorem proved independently by Lin \cite{Lin02} and Dadarlat--Eilers \cite{DadarlatEilers01, DadarlatEilers02}.
It states that if a Cuntz pair $(\phi, \psi)$ represents the zero element in $KK(A,B)$, then $\phi$ and $\psi$ are homotopic via a unitary path in $\eins+B\otimes \CK$ after stabilizing both $\phi$ and $\psi$ with a suitable $\ast$-homomorphism.
(This is a tiny oversimplification. It captures how we reprove the original theorem as a special case of ours, but the former is actually a bit weaker.
This boils down to the difference between \emph{proper} and \emph{strong} asymptotic unitary equivalence in \autoref{def:properasym}.)

While the stable uniqueness theorem was originally motivated as a tool for classification of nuclear \cstar-algebras \cite{Lin02, DadarlatEilers02}, the past decade has found other groundbreaking utilizations of the theorem, such as the Quasidiagonality Theorem of Tikuisis--White--Winter \cite{TikuisisWhiteWinter17} and Schafhauser's AF embeddability theorem \cite{Schafhauser20}.
It is worth pointing out explicitly that the celebrated Kirchberg--Phillips theorem \cite{Kirchberg94, Phillips00}, which classifies separable simple nuclear purely infinite \cstar-algebras via $KK$-theory, is known to admit a particularly slick proof with the help of the stable uniqueness theorem; see \cite{Gabe24} for the most recent one due to the first named author.

With the classification of simple nuclear well-behaved \cstar-algebras being essentially complete due to the work of many hands --- an incomplete list of references being \cite{Kirchberg94, Phillips00, GongLinNiu20, GongLinNiu20_2, ElliottGongLinNiu15, TikuisisWhiteWinter17, ElliottGongLinNiu20, GongLin20, GongLin22, CGSTW23} --- a natural future research direction is the classification of \cstar-dynamical systems.
In the most general sense, this concerns \cstar-algebras equipped with a continuous action of a locally compact group.
At least for special choices of the acting group (such as finite groups or the integers), classification results in this vein have been investigated since the 1980s.
After some preliminary prototype arguments by Herman--Jones \cite{HermanJones82} and Herman--Ocneanu \cite{HermanOcneanu84} inspired by work of Connes, the concept of the Rokhlin property was fleshed out and exploited for classification by Kishimoto \cite{Kishimoto95, Kishimoto98}, Izumi \cite{Izumi04}, and many others.
By now these related techniques have been applied many times and also extended to actions of compact groups \cite{HirshbergWinter07, Gardella19, Gardella21}, actions of $\IR$ \cite{Kishimoto96_R, Szabo21R}, and more complicated infinite discrete groups \cite{Izumi10}.
Generally speaking, classification via a Rokhlin-type property dominates the bulk of the available literature when it comes to methodology, but often comes at the cost of restrictions on the actions that one ends up classifying.
It would appear that the latter drawback is the least pronounced for discrete groups that are in a suitable way built out of the integer group $\IZ$.
This led to the recent breakthrough of Izumi--Matui \cite{IzumiMatui20, IzumiMatui21, IzumiMatui21_2} who managed to classify outer actions of poly-$\IZ$ groups on Kirchberg algebras with the help of invariants that not only seem to resemble $K$-theory, but were later argued by Meyer \cite{Meyer21} to genuinely amount to equivariant $KK$-theory.

The fact that $KK$-theory is extendible to \cstar-dynamics instead of just \cstar-algebras is not a mere aesthetic coincidence, but was the key to the impactful applications of Kasparov's ideas to classically motivated problems such as the Novikov conjecture.
In complete analogy to the Kirchberg--Phillips theorem mentioned above, Phillips has hypothesized in the past that outer actions of finite groups on Kirchberg algebras ought to be classified by equivariant $KK$-theory.
Various authors have subsequently speculated (albeit with little trace in the published literature) that this ought to be true for actions of countable amenable groups.
Some evidence for this was provided in \cite{Szabo18kp}, which was recently generalized by Suzuki \cite{Suzuki21}, indicating that there may even be hope to classify amenable actions of not necessarily amenable groups.
In contrast to the methodology present in the aforementioned Elliott program, however, it has in large part remained a mystery how the information encoded in equivariant $KK$-theory can be systematically employed towards the classification of actions.
The impressive recent work of Izumi--Matui mentioned above, for example, uses bundle-like invariants that could only in hindsight be interpreted in $K$-theoretical language (for groups without torsion) as a result of the categorical insights related to the Baum--Connes conjecture \cite{MeyerNest06}.
The intention behind this article is to build a systematic machinery to utilize equivariant $KK$-theory towards the classification of \cstar-dynamics.
We expect and hope for it to lead to a paradigm shift in the direction of research that concerns dynamical generalizations of the Elliott classification program.

Since the methodology of Izumi--Matui is closely tied to the Rokhlin property and the Evans--Kishimoto intertwining argument \cite{EvansKishimoto97}, it is increasingly difficult to implement for more complicated groups, and in fact is too restrictive for groups with torsion.
Recently, the second named author proposed a categorical framework for \cstar-dynamics \cite{Szabo21cc} to classify group actions up to cocycle conjugacy based on an Elliott intertwining argument that is intended to provide an alternative to the one of Evans--Kishimoto.
In this framework, arrows between actions of \cstar-algebras are so-called cocycle morphisms.
In analogy to ordinary \cstar-algebra classification, it is an important intermediate step to solve the \emph{uniqueness problem}, i.e., to determine in terms of classifying invariants when two cocycle morphisms are \emph{approximately/asymptotically unitarily equivalent}; for more details see \cite{Szabo21cc}.
Since it was also argued that equivariant $KK$-theory can be viewed as a (bi-)functor on this enlarged category, it is natural expect that it ought to represent one of the key obstructions to solving said existence/uniqueness problem.

In \cite{Thomsen98}, Thomsen showed that equivariant $KK$-theory can be expressed using Cuntz pairs if one considers pairs of cocycle representations; we will henceforth refer to this framework as the Cuntz--Thomsen picture. 
That is, instead of considering equivariant $\ast$-homomorphisms, one considers $\ast$-homomorphisms $A\to \CM(B)$ that are equivariant after twisting the action on $B$ with distinguished cocycles.
In this picture, Thomsen proved that equivariant $KK$-theory can be described by a Higson-type universal property.
The aim of this paper is to prove the following stable uniqueness theorem for equivariant $KK$-theory. 
For the most general version, see \autoref{thm:main}.

\begin{theoreme} 
Let $A$ and $B$ be separable \cstar-algebras, and let $G$ be a second-countable locally compact group. 
Let $\alpha: G\curvearrowright A$ and $\beta: G\curvearrowright B$ be two continuous actions.
Let
\[
(\phi,\Iu), (\psi,\Iv): (A,\alpha)\to (\CM(B\otimes\CK),\beta\otimes\id_\CK)
\]
be a pair of cocycle representations forming an anchored $(\alpha, \beta)$-Cuntz pair (see \autoref{def:anchored}).
Then $\big[ (\phi, \Iu), (\psi, \Iv)\big] = 0$ in $KK^G(\alpha, \beta)$ if and only if there exists a third cocycle representation $(\theta,\Iy)$ and a norm-continuous path $u\colon [0,\infty) \to \mathcal U(\eins+M_2(B \otimes \CK))$ with $u_0 = \eins$ such that
\[
\lim_{t\to \infty} \left\| \begin{psmallmatrix} \psi(a) & 0 \\ 0 & \theta(a) \end{psmallmatrix} - u_t \begin{psmallmatrix} \phi(a) & 0 \\ 0 & \theta(a) \end{psmallmatrix} u_t^\ast \right\| = 0 
\]
for all $a\in A$ and
\[
\lim_{t\to \infty} \max_{g\in K} \left\| \begin{psmallmatrix} \Iv_g & 0 \\ 0 & \Iy_g \end{psmallmatrix} - u_t \begin{psmallmatrix} \Iu_g & 0 \\ 0 & \Iy_g \end{psmallmatrix} (\beta \otimes \id_\CK)_g^{(2)} (u_t)^\ast \right\| =0
\]
for all compact sets $K \subseteq G$.
\end{theoreme}

In a follow-up paper \cite{GabeSzabo24kp}, the above main result is employed to prove a dynamical Kirchberg--Phillips theorem, i.e., a classification theorem for certain amenable group actions on Kirchberg algebras via equivariant Kasparov theory.
In the special case of discrete groups, this classifies all amenable and outer actions, which completely settles an important open problem mentioned above.

Besides this intended application, we wish to emphasize some novel aspects in our approach towards the theorem, which we believe might make this article appealing even for those readers who are only interested in the known (non-dynamical) stable uniqueness theorem of Lin and Dadarlat--Eilers.
If we were to summarize it in an oversimplified slogan, it would be that upon direct comparison, we employ certain tricks related to quasicentral approximate units in order to overcome all of the truly substantial or deep technical obstacles that the known proof in the literature was dealing with.
For those readers familiar with the original proof, we point out that our approach is self-contained in the Cuntz--Thomsen picture, without a need to translate certain problems into the Kasparov picture or Fredholm picture (\cite{Higson87}) of $KK$-theory to solve them.
Furthermore, we found a more direct alternative to the part about automorphisms of \cstar-algebras that are norm-homotopic to the identity, which was utilizing the theory of derivations in a key way.

The article is organized as follows.
In the first preliminary section, we introduce basic notation, recall needed definitions or arguments from the literature, and introduce equivariant $KK$-theory in the Cuntz--Thomsen picture.
In the second section, we introduce a Cuntz--Thomsen analog of the equivalence relation of operator homotopy.
A pair of cocycle representations $(\phi,\Iu), (\psi,\Iv): (A,\alpha)\to (\CM(B\otimes\CK),\beta\otimes\id_\CK)$ is said to be operator homotopic, if it is possible to find a continuous path of unitaries $\set{u_t}_{0\leq t\leq 1}$ with $u_0=\eins$ and $\ad(u_1)\circ(\phi,\Iu)=(\psi,\Iv)$, and such that $\ad(u_t)\circ(\phi,\Iu)$ forms a Cuntz pair with $(\psi,\Iv)$ for all $t\in [0,1]$.
This entails that the two cocycle representations had to form a Cuntz pair to begin with, and that their associated $KK$-class vanishes.
Although this is in general a strict implication, we show as our main result in the second section (\autoref{thm:homotopy-implies-operator-homotopy}) that anchored Cuntz pairs with vanishing $KK$-class can always be arranged to become operator homotopic after stabilizing with a suitably chosen cocycle representation. 
In the third section, we investigate sufficient conditions to determine when cocycle representations absorb each other in the spirit of \cite{Kasparov80} and Voiculescu's theorem \cite{Voiculescu76}.
We introduce the notion of weak containment between cocycle representations, which unifies weak containment of unitary representations with certain known subequivalence relations for $*$-homomorphisms of \cstar-algebras.
The most appealing feature of this concept is that a cocycle representation $(\phi,\Iu)$ turns out to weakly contain another cocycle representation $(\psi,\Iv)$ precisely when the infinite repeat $(\phi^\infty,\Iu^\infty)$ absorbs the infinite repeat $(\psi^\infty,\Iv^\infty)$ (see \autoref{def:inf-repeat}).
Building on this fact, we give a new elementary proof of the fact that for separable $A$ and $B$, it is always possible to find a cocycle representation $(\theta,\Iy)$ that absorbs every other one; see \autoref{thm:abs-rep-existence}.
This generalizes various similar results from the literature \cite{Thomsen01, Thomsen05, GabeRuiz15}.
Our proof has the advantage that it is easily adaptable in a more general context, as well as being based on a simplified argument that skips the previously common part about Kasparov--Stinespring dilations of completely positive maps.
In the fourth section, we investigate criteria on pairs of cocycle representations to be strongly asymptotically unitarily equivalent.
Although a priori much weaker, we prove that it is sufficient to ask for the asymptotic unitary equivalence to be implemented by certain paths of unitaries in the multiplier algebra $\CM(B)$ instead of $\eins+B$.
In comparison to the previously known proof in the non-dynamical setting that utilized the theory of derivations in this step, our proof combines a functional calculus argument with the existence of quasicentral approximate units inside $B$.
As a consequence of this observation, we deduce as the main result of the fourth section (\autoref{cor:op-hom-implies-saue}) that operator homotopic cocycle representations are strongly asymptotically unitarily equivalent.
In the fifth and final section, we combine the main results of the other sections and deduce our final main result \autoref{thm:main}.
\smallskip

\textbf{Acknowledgements.} 
The first named author has been supported by the Independent Research Fund Denmark through the Sapere Aude:\ DFF-Starting Grant 1054-00094B.
The second named author has been supported by the start-up grant STG/18/019 of KU Leuven, the research project C14/19/088 funded by the research council of KU Leuven, and the project G085020N funded by the Research Foundation Flanders (FWO).

We would like to express our gratitude to the anonymous referees for many helpful comments and suggestions that improved the article.


\section{Preliminaries}

\begin{nota}
Throughout, $G$ will denote a second-countable, locally compact group unless specified otherwise.
Normal capital letters like $A,B,C$ will denote generic \cstar-algebras. 
The multiplier algebra of $A$ is denoted as $\CM(A)$, whereas $A^\dagger$ denotes the proper unitization of $A$, i.e., one adds a new unit even if $A$ was unital.
For a unital \cstar-algebra $A$, we write $\CU(A)$ for its unitary group and $\CU_0(A)$ for the connected component of the unit element inside the unitary group.
For a (not necessarily unital) \cstar-algebra, we write $\CU(\eins+A)$ for the set of all unitaries in $A^\dagger$ whose scalar part is $1$, which can be canonically identified with $\CU(A)$ if $A$ was already unital.
Throughout the article, the symbol $\CK$ denotes the \cstar-algebra of compact operators on a separable infinite-dimensional Hilbert space.
If a particular statement calls for specifying the Hilbert space $\CH$, we write $\CK(\CH)$.
Greek letters such as $\alpha,\beta,\gamma$ are used for point-norm continuous maps $G\to\Aut(A)$, in particular for point-norm continuous $G$-actions.
In this case we use the same symbol $\alpha$ for the induced action $\alpha: G\to\Aut(\CM(A))$, which is point-strictly continuous, but may in general fail to be point-norm continuous.
Depending on the situation, we may denote $\id_A$ either for the identity map on $A$ or the trivial $G$-action on $A$.
We will denote by $A^\alpha$ or $\CM(A)^\alpha$ the \cstar-subalgebra of fixed points (in $A$ or $\CM(A)$) with respect to $\alpha$.
Normal alphabetical letters such as $u,v,U,V$ are used for unitary elements in some \cstar-algebra $A$.
If either $u\in\CU(\CM(A))$ or $u\in\CU(\eins+A)$, we denote by $\ad(u)$ the induced inner automorphism of $A$ given by $a\mapsto uau^*$.
Double-struck letters such as $\Iu, \Iv, \IU, \IV$ are used for strictly continuous maps $G\to\CU(\CM(A))$.
Most of the time they will be assumed to be (1-)cocycles with respect to an action $\alpha: G\curvearrowright A$, which for the map $\Iu$ would mean that it satisfies the cocycle identity $\Iu_{gh}=\Iu_g\alpha_g(\Iu_h)$ for all $g,h\in G$.
Under this assumption, one obtains a new (cocycle perturbed) action $\alpha^\Iu: G\curvearrowright A$ via $\alpha^\Iu_g=\ad(\Iu_g)\circ\alpha_g$.
\end{nota}

\begin{defi}[see {\cite[Section 1]{Szabo21cc}}]
Let $\alpha: G\curvearrowright A$ and $\beta: G\curvearrowright B$ be two actions on \cstar-algebras.
\begin{enumerate}[label=\textup{(\roman*)}]
\item A \emph{cocycle representation} $(\phi,\Iu): (A,\alpha)\to (\CM(B),\beta)$ consists of a $*$-homomorphism $\phi: A\to\CM(B)$ and a strictly continuous $\beta$-cocycle $\Iu: G\to\CU(\CM(B))$ satisfying $\ad(\Iu_g)\circ\beta_g\circ\phi=\phi\circ\alpha_g$ for all $g\in G$.
\item If additionally $\phi(A)\subseteq B$, then the pair $(\phi,\Iu)$ is called a \emph{cocycle morphism}, and we denote $(\phi,\Iu): (A,\alpha)\to (B,\beta)$.
\end{enumerate}
In the situation above, if $\phi=0$ and $\Iu=\eins$, then $(\phi,\Iu)$ is called the zero representation, and on occasion we write $(0,\eins)=0$ where notationally convenient.
\end{defi}

\begin{nota}
Let us say that an action $\beta: G\curvearrowright B$ on a \cstar-algebra is \emph{strongly stable} if $(B,\beta)$ is (genuinely) conjugate to $(B\otimes\CK,\beta\otimes\id_\CK)$.
\end{nota}

\begin{rem}
We will repeatedly and without mention use the fact that an action $\beta: G\curvearrowright B$ is strongly stable if and only if there is a sequence of isometries $r_n\in\CM(B)^\beta$ such that $\eins=\sum_{n=1}^\infty r_nr_n^*$ in the strict topology.
The ``only if'' part is clear since such a sequence can be obtained through an inclusion $\CB(\ell^2(\IN))\cong\CM(\CK)\subseteq\CM(B\otimes\CK)^{\beta\otimes\id_\CK}$.
For the ``if'' part, one realizes that if $\{e_{k,\ell}\}_{k,\ell\geq 1}$ is a set of matrix units generating the compacts $\CK$, then
\[
B\otimes\CK\to B,\quad b\otimes e_{k,\ell}\mapsto r_kbr_\ell^*
\]
defines an isomorphism that is equivariant with respect to $\beta\otimes\id_\CK$ and $\beta$.
\end{rem}

We shall now recall some necessary background on equivariant $KK$-theory.
Throughout the paper the focus lies on the Cuntz--Thomsen picture \cite{Cuntz83, Cuntz84, Higson87, Thomsen98} rather than Kasparov's original picture \cite{Kasparov88}.

\begin{defi}[cf.\ {\cite[Section 3]{Thomsen98}}] \label{def:equi-Cuntz-pair}
Let $\alpha: G\curvearrowright A$ and $\beta: G\curvearrowright B$ be two actions on  \cstar-algebras where $A$ is separable and $B$ is $\sigma$-unital.
An $(\alpha,\beta)$-\emph{Cuntz pair} is a pair of cocycle representations
\[
(\phi,\Iu), (\psi,\Iv): (A,\alpha) \to (\CM(B\otimes\CK),\beta\otimes\id_\CK),
\]
such that the pointwise differences $\phi-\psi$ and $\Iu-\Iv$ take values in $B\otimes\CK$.
(In Thomsen's work it was also assumed that the map $\Iu-\Iv$ is norm-continuous. This turns out to be redundant, see \cite[Proposition 6.9]{Szabo21cc}.)
If $\beta$ is assumed to be strongly stable, then we also allow $(\CM(B),\beta)$ as the codomain of an $(\alpha,\beta)$-Cuntz pair for notational convenience.
\end{defi}

\begin{defi}[cf.\ {\cite[Lemma 3.4]{Thomsen98}}]
Let $\beta: G\curvearrowright B$ be an action on a \cstar-algebra.
Suppose that there exists a unital inclusion $\CO_2\subseteq\CM(B)^\beta$.
For two isometries $t_1,t_2\in\CM(B)^\beta$ with $t_1t_1^*+t_2t_2^*=\eins$, we may consider the $\beta$-equivariant $*$-homomorphism
\[
\CM(B)\oplus\CM(B)\to\CM(B),\quad b_1\oplus b_2\mapsto b_1\oplus_{t_1,t_2} b_2 := t_1b_1t_1^*+t_2b_2t_2^*.
\]
Up to unitary equivalence with a unitary in $\CM(B)^\beta$, this $*$-homomorphism does not depend on the choice of $t_1$ and $t_2$:
If $v_1, v_2\in\CM(B)^\beta$ are two other isometries with $v_1v_1^*+v_2v_2^*=\eins$, then the unitary equivalence between ``$\oplus_{t_1,t_2}$'' and ``$\oplus_{v_1,v_2}$'' is implemented by $w=t_1v_1^*+t_2v_2^*\in\CM(B)^\beta$.
One refers to the element $b_1\oplus_{t_1,t_2} b_2$ as the \emph{Cuntz sum} of the two elements $b_1$ and $b_2$ (with respect to $t_1$ and $t_2$).

Now let $\alpha: G\curvearrowright A$ be another action on a \cstar-algebra, and $(\phi,\Iu), (\psi,\Iv): (A,\alpha)\to (\CM(B),\beta)$ two cocycle representations.
We likewise define the (pointwise) Cuntz sum
\[
(\phi,\Iu)\oplus_{t_1,t_2} (\psi,\Iv) = (\phi\oplus_{t_1,t_2}\psi, \Iu\oplus_{t_1,t_2}\Iv): (A,\alpha)\to (\CM(B),\beta),
\]
which is easily seen to be another cocycle representation.
Since its unitary equivalence class does not depend on the choice of $t_1$ and $t_2$, we will often omit $t_1$ and $t_2$ from the notation if it is clear from context that a given statement is invariant under said equivalence.
\end{defi}

\begin{nota}
Given a \cstar-algebra $B$, we denote $B[0,1]=\CC[0,1]\otimes B$.
If one has an action $\beta: G\curvearrowright B$, we consider the obvious $G$-action on $B[0,1]$ given by $\beta[0,1]=\id_{\CC[0,1]}\otimes\beta$.
\end{nota}

\begin{defi}[see {\cite[Section 3]{Thomsen98}}] \label{def:KKG-Thomsen}
Let $A$ be a separable \cstar-algebra and $B$ a $\sigma$-unital \cstar-algebra.
For two actions $\alpha: G\curvearrowright A$ and $\beta: G\curvearrowright B$,
let $\IE^G(\alpha,\beta)$ denote the set of all $(\alpha,\beta)$-Cuntz pairs, and let $\IDh^G(\alpha, \beta)$ denote the subset of all \emph{degenerate} $(\alpha, \beta)$-Cuntz pairs, i.e., those with $\phi=\psi$ and $\Iu = \Iv$.
A \emph{cocycle pair} is a Cuntz pair $\big( (\phi,\Iu), (\psi,\Iv) \big) \in \IE^G(\alpha, \beta)$ with $\phi=\psi=0$.
(We note that Thomsen refers to cocycle pairs as degenerate Cuntz pairs.
We have chosen to rename these, since this does not coincide with the usual notion. We will be working with a notion of degenerate Cuntz pairs that better resembles the usual notion for Kasparov modules. We point out that initial confusion about this matter led the second named author to give a slightly wrong definition of $KK^G$ in \cite[Section 6]{Szabo21cc}. Fortunately this mistake has no bearing on the validity of the results therein.)
We will slightly abuse notation and denote a cocycle pair of the form $\big( (0, \Iu), (0, \Iv)\big)$ by $(\Iu, \Iv)$. 
In the special case of having the trivial example $\big((\phi,\Iu), (\psi, \Iv)\big) = \big((0, \eins), (0, \eins)\big)$, we denote its associated degenerate Cuntz pair simply by the symbol $0$.

Two elements $x_0, x_1\in\IE^G(\alpha,\beta)$ are called \emph{homotopic} if there exists an $(\alpha,\beta[0,1])$-Cuntz pair in $\IE^G(\alpha,\beta[0,1])$ which restricts to $x_0$ upon evaluation at $0\in [0,1]$, and restricts to $x_1$ in $\IE^G(\alpha,\beta)$ upon evaluation at $1\in [0,1]$.
Let us for the moment write $x_0\sim_h x_1$.

For any unital inclusion $\CO_2\subseteq\CM(B\otimes\CK)^{\beta\otimes\id_\CK}$ with generating isometries $t_1,t_2$, one can perform the Cuntz addition for two $(\alpha,\beta)$-Cuntz pairs as
\[
\begin{array}{cl}
\multicolumn{2}{l}{
\big( (\phi^0,\Iu^0), (\psi^0,\Iv^0) \big) \oplus_{t_1,t_2} \big( (\phi^1,\Iu^1), (\psi^1,\Iv^1) \big) } \\
=& \big( (\phi^0,\Iu^0)\oplus_{t_1,t_2}(\phi^1,\Iu^1), (\psi^0,\Iv^0)\oplus_{t_1,t_2} (\psi^1,\Iv^1) \big).
\end{array}
\]
This is independent of the choice of $t_1,t_2$ up to homotopy; see \cite[Lemma 3.3]{Thomsen00}.
It was proved by Thomsen \cite[Theorem 3.5]{Thomsen98} that Kasparov's equivariant $KK$-group $KK^G(\alpha,\beta)$  is naturally isomorphic to $\IE^G(\alpha,\beta)/\!\!\sim$, whereby one has $x\sim y$ if there exist cocycle pairs $d_1,d_2\in\IE^G(\alpha,\beta)$ such that $x\oplus d_1\sim_h y\oplus d_2$. 
For an $(\alpha,\beta)$-Cuntz pair consisting of $(\phi,\Iu)$ and $(\psi,\Iv)$, we denote its associated equivalence class in $KK^G(\alpha,\beta)$ by $[(\phi,\Iu),(\psi,\Iv)]$.
\end{defi}

We will prove a cancellation result for certain Cuntz pairs that provides a slightly different picture of $KK^G$ (\autoref{prop:KKG-homotopy}) without having to stabilize with cocycle pairs as in Thomsen's description of $KK^G$.
For the readers unfamiliar with Cuntz's picture of $KK$-theory, we include some self-contained (though not too elaborate) arguments for the reader's convenience proving the abelian group structure for the classes of Cuntz pairs.

\begin{prop}[cf.\ {\cite[Lemma 1.3.6]{JensenThomsen}}] \label{prop:isometric-path}
There exist strictly continuous maps of isometries $S_1: [0,1)\to\CM(\CK)$ and $S_2: (0,1]\to\CM(\CK)$ satisfying $S_1^{(0)}=\eins$, $S_2^{(1)}=\eins$, and $S_1^{(t)}S_1^{(t)*}+S_2^{(t)}S_2^{(t)*}=\eins$ for all $t\in (0,1)$.
\end{prop}
\begin{proof}
After identifying $\CM(\CK)$ with the algebra of bounded operators on the Hilbert space $L^2[0,1]$ (with respect to the Lebesgue measure), one defines for all $\xi\in L^2[0,1]$ and $s\in [0,1]$
\[
S_1^{(t)}(\xi)(s)=\begin{cases} \frac{\xi((1-t)^{-1} s)}{\sqrt{1-t}} &,\quad s\leq 1-t \\ 0 &,\quad s>1-t, \end{cases} 
\]
and
\[
S_2^{(t)}(\xi)(s)= \begin{cases} \frac{\xi(t^{-1}(s+t-1))}{\sqrt{t}} &,\quad s\geq 1-t \\ 0&,\quad s<1-t. \end{cases}
\]
Then it is easy to check that these define isometries whenever $t\in [0,1]$ is so that either one of these operators is defined.
Evidently one has $S_1^{(0)}=\eins$ and $S_2^{(1)}=\eins$.
The range of $S_1^{(t)}$ is the subspace $L^2[0,1-t]$ for $t<1$ and the range of $S_2^{(t)}$ is the subspace $L^2[1-t,1]$ for $t>0$.
The claim follows.
\end{proof}

\begin{lemma} \label{lem:equivalence-simplified}
Let $A$ be a separable \cstar-algebra and $B$ a $\sigma$-unital \cstar-algebra.
Let $\alpha: G\curvearrowright A$ and $\beta: G\curvearrowright B$ be two actions.
\begin{enumerate}[label=\textup{(\roman*)},leftmargin=*]
\item For every $x\in \IE^G(\alpha, \beta)$ and $d\in \ID^G(\alpha, \beta)$, $x\oplus d$ is homotopic to $x$. \label{lem:equivalence-simplified:1}
\item The quotient $\IE^G(\alpha,\beta)/{\sim_h}$ is an abelian group under Cuntz addition.
The neutral element is represented by any element in $\ID^G(\alpha,\beta)$. \label{lem:equivalence-simplified:2}
\end{enumerate}
\end{lemma}
\begin{proof}
For notational convenience, we shall assume that $\beta$ is strongly stable.

\ref{lem:equivalence-simplified:1}:
Choose three cocycle representations
\[
(\phi, \Iu), (\psi, \Iv), (\theta,\Ix): (A,\alpha)\to (\CM(B),\beta)
\]
in order to write $x=\big( (\phi, \Iu), (\psi, \Iv) \big)$ and $d=\big( (\theta,\Ix), (\theta,\Ix) \big)$.
Using \autoref{prop:isometric-path}, we choose strictly continuous maps of isometries $S_1: [0,1)\to\CM(\CK)\subseteq \CM(B)^{\beta}$ and $S_2: (0,1]\to\CM(\CK)\subseteq \CM(B)^{\beta}$ satisfying $S_1^{(0)}=\eins$, $S_2^{(1)}=\eins$, and $S_1^{(t)}S_1^{(t)*}+S_2^{(t)}S_2^{(t)*}=\eins$ for all $t\in (0,1)$.
For each $t\in (0,1]$, we may consider the element
\[
\IE^G(\alpha,\beta) \ni x_t = \big( (\phi, \Iu) \oplus_{S_1^{(t/2)},S_2^{(t/2)}} (\theta,\Ix), (\psi, \Iv)\oplus_{S_1^{(t/2)},S_2^{(t/2)}} (\theta,\Ix) \big).
\]
Because of the strict continuity of the involved maps and $S_1^{(0)}=\eins$, it also follows that $S_2^{(t)}S_2^{(t)*}\to 0$ strictly as $t\to 0$.
This continuous family of Cuntz pairs thus converges to $x_0=x$ as $t\to 0$.
This provides a homotopy between $x=x_0$ and $x_1=x\oplus_{S_1^{(1/2)},S_2^{(1/2)}} d$, proving the claim.

\ref{lem:equivalence-simplified:2}:
Since it is somewhat clear from earlier remarks, we will omit the proof that Cuntz addition descends to a well-defined associative binary operation on $\IE^G(\alpha,\beta)/{\sim_h}$.
We may already conclude from the first part that every degenerate element induces a (right) neutral element in the semigroup $\IE^G(\alpha, \beta)/{\sim_h}$.
It suffices to prove that for any $x,y\in\IE^G(\alpha,\beta)$, one has $x\oplus y\sim_h y\oplus x$ and that there exists $x'\in\IE^G(\alpha,\beta)$ such that $x\oplus x'$ is homotopic to an element in $\ID^G(\alpha,\beta)$.

Choose isometries $r_1, r_2\in\CM(B)^{\beta}$ with $r_1r_1^*+r_2r_2^*=\eins$, and let us agree on forming Cuntz sums with this specific pair.
Define for $t\in [0,1]$ the isometries $r_1^{(t)}=(1-t)^{1/2} r_1+t^{1/2} r_2$ and $r_2^{(t)}=-t^{1/2}r_1+(1-t)^{1/2}r_2$, which also satisfy the $\CO_2$-relation.
These define norm-continuous maps, and we may hence observe that
\[
x\oplus_{r_1,r_2} y = x\oplus_{r_1^{(0)},r_2^{(0)}} y \sim_h x\oplus_{r_1^{(1)},r_2^{(1)}} y = x \oplus_{r_2,r_1} y = y\oplus_{r_1,r_2} x.
\]
If we write $x=\big( (\phi,\Iu), (\psi,\Iv) \big)$, we claim that $x'=\big( (\psi,\Iv), (\phi,\Iu) \big)$ does the trick.
Indeed, the homotopy between $x\oplus_{r_1,r_2} x'$ to an element in $\ID^G(\alpha,\beta)$ is witnessed by the element $X\in\IE^G(\alpha,\beta[0,1])$ given by
\[
X_t = \Big( (\phi,\Iu)\oplus_{r_1,r_2} (\psi,\Iv), (\psi,\Iv)\oplus_{r_1^{(t)},r_2^{(t)}} (\phi,\Iu) \Big) ,\quad t\in [0,1].
\]
Computing that this is indeed a Cuntz pair is straightforward.
\end{proof}

\begin{defi}\label{def:anchored}
We say that an $(\alpha, \beta)$-Cuntz pair $\big( (\phi, \Iu), (\psi,\Iv) \big)$ is \emph{anchored} if the associated cocycle pair $(\Iu, \Iv)$ is homotopic to $(\eins, \eins)$.
Let $\IE_0^G(\alpha,\beta)$ denote the set of all anchored Cuntz pairs.
\end{defi}

In the above definition, we picture in our mind's eye the cocycle pair $(\eins , \eins)$ as the anchor, and the homotopy of cocycle pairs as the chain connecting our Cuntz pair to the anchor.
Note that such a homotopy from $\big( (0, \Iu), (0, \Iv) \big)$ to $\big( (0, \eins), (0,\eins)\big)$ can be chosen such that the entire homotopy is of the form $[0,1] \ni s\mapsto \big( (0, \Iu_s), (0, \Iv_s) \big)$ by simply taking any homotopy and considering the induced cocycle pair.

\begin{prop} \label{prop:KKG-homotopy}
Let $\pi: \IE^G(\alpha,\beta)/{\sim}\to KK^G(\alpha,\beta)$ be Thomsen's natural isomorphism from \cite[Theorem 3.5]{Thomsen98}.
\begin{enumerate}[label=\textup{(\roman*)},leftmargin=*]
\item The classes of cocycle pairs inside $\IE^G(\alpha,\beta)/{\sim_h}$ form a subgroup. \label{prop:KKG-homotopy:1}
\item The inclusion $\IE_0^G(\alpha,\beta)\subseteq\IE^G(\alpha,\beta)$ induces an isomorphism of abelian groups $\IE_0^G(\alpha,\beta)/{\sim_h}\cong\IE^G(\alpha,\beta)/{\sim}$. \label{prop:KKG-homotopy:2}
\item For $x,y\in\IE_0^G(\alpha,\beta)$, one has $\pi([x])=\pi([y])$ if and only if $x\sim_h y$. \label{prop:KKG-homotopy:3}
\end{enumerate}
In summary, $KK^G(\alpha, \beta)$ can be canonically identified with the set of homotopy classes of anchored $(\alpha, \beta)$-Cuntz pairs.
\end{prop}
\begin{proof}
Part \ref{prop:KKG-homotopy:1} follows because cocycle pairs are closed under Cuntz sums and inverses, the latter being evident due to the proof of \autoref{lem:equivalence-simplified}. 
Part \ref{prop:KKG-homotopy:3} is a direct consequence of \ref{prop:KKG-homotopy:2}.

\ref{prop:KKG-homotopy:2}: By the definition of the relation $\sim$ and \autoref{lem:equivalence-simplified}, the quotient $\IE^G(\alpha,\beta)/{\sim}$ is nothing but the group quotient of $\IE^G(\alpha,\beta)/{\sim_h}$ modulo the subgroup $H_\beta$ spanned by (homotopy classes of) cocycle pairs.
We have an idempotent group endomorphism $p: \IE^G(\alpha,\beta)/{\sim_h}\to H_\beta$ induced by the assignment $\big( (\phi,\Iu), (\psi,\Iv) \big)\mapsto \big( (0,\Iu),(0,\Iv) \big)$.
By definition, $\IE_0^G(\alpha,\beta)$ is the set of those elements in $\IE^G(\alpha,\beta)$ whose homotopy class belongs to the kernel of $p$.
By the basic algebra of abelian groups, the claim follows.
\end{proof}

The following argument is most likely well-known, but we record it here as it will come in handy later.

\begin{prop} \label{prop:Cuntz-pair-small-perturbation}
Let $(\phi,\Iu), (\psi,\Iv): (A,\alpha)\to(\CM(B\otimes\CK),\beta\otimes\id_\CK)$ be two cocycle representations that form an $(\alpha,\beta)$-Cuntz pair.
For any unitary $u\in\CU(\eins+B\otimes\CK)$, one has  that $\big( (\phi,\Iu), \ad(u)\circ(\psi,\Iv) \big)$ is homotopic to $\big( (\phi,\Iu), (\psi,\Iv) \big)$.
\end{prop}
\begin{proof}
By \autoref{lem:equivalence-simplified}, there is a homotopy from $\big( (\phi,\Iu), \ad(u)\circ(\psi,\Iv) \big)$ to
\[
\big( (\phi,\Iu), \ad(u)\circ(\psi,\Iv) \big) \oplus 0 =  \big( (\phi\oplus 0,\Iu\oplus\eins), \ad(u\oplus u^*)\circ(\psi\oplus 0,\Iv\oplus\eins) \big).
\]
The unitary $u\oplus u^*\in\CU(\eins+B\otimes\CK)$ is norm-homotopic to the unit inside this unitary group; for instance, if the Cuntz sum is formed via $r_1,r_2$, then the assignment
\[
[0,1] \ni t \mapsto \eins-t +tr_1ur_1^*+\sqrt{t(1-t)}(r_2(u^*-\eins)r_1^*-r_1(u-\eins)r_2^*)+tr_2u^*r_2^*
\]
is such a homotopy.
This yields a homotopy between
\[
\big( (\phi\oplus 0,\Iu\oplus\eins),  \ad(u\oplus u^*)\circ (\psi\oplus 0,\Iv\oplus\eins) \big) \quad \textrm{and} \quad \big( (\phi\oplus 0,\Iu\oplus\eins),   (\psi\oplus 0,\Iv\oplus\eins) \big).
\] 
The latter Cuntz pair is homotopic to $\big( (\phi,\Iu), (\psi,\Iv) \big)$ by \autoref{lem:equivalence-simplified}.
\end{proof}

We will in various instances make use of the following useful fact due to Kasparov concerning quasicentral approximate units that are approximately invariant under a group action.

\begin{lemma}[see {\cite[Lemma 1.4]{Kasparov88}}] \label{lem:Kasparov}
Let $\beta: G\curvearrowright B$ be an action on a $\sigma$-unital \cstar-algebra.
Then for any separable \cstar-subalgebra $D\subseteq\CM(B)$, there exists a countable, increasing approximate unit of positive contractions $h_n\in B$ 
satisfying  
\[
\lim_{n\to\infty} \|[h_n,d]\|=0
\]
for all $d\in D$, and
\[
\lim_{n\to\infty} \max_{g\in K} \| h_n-\beta_g(h_n) \| = 0
\] 
for all compact sets $K\subseteq G$.
\end{lemma}

To conclude this section, we establish the following fact for later use.
Although it is probably known to some experts, we decided to include an elementary and self-contained proof for the reader's convenience.
The argument itself is adapted from the proof of \cite[Lemma 3.26]{CuntzMeyerRosenberg}.

\begin{theorem} \label{thm:basic-op-homotopy}
Let $\CH$ be an infinite-dimensional separable Hilbert space.
There exists a unital representation $\theta: \CC[0,1]\to \CB(\CH)$ and a norm-continuous unitary path $w: [0,1]\to\CU(\IC\oplus\CH)$ with $w_0=\eins$ such that
\begin{equation}\label{eq:ev1theta}
\ev_1\oplus\theta = \ad(w_1)\circ(\ev_0\oplus\theta)
\end{equation}
and
\begin{equation}\label{eq:wtcommute}
\big[ w_t, (\ev_0\oplus\theta)(\CC[0,1]) \big] \subseteq \CK(\IC\oplus\CH),\quad t\in [0,1].
\end{equation}
Here the direct sum refers to the ordinary direct sum of $*$-representations.
\end{theorem}
\begin{proof}
We note that since the Hilbert space $\CH$ is uniquely determined up to isometric isomorphism, it suffices to prove the claim for a concrete choice of $\CH$, which we will specify below.
We will construct three self-adjoint unitaries $w^{(i)}\in \mathcal U(\mathbb C \oplus \mathcal H)$ such that 
\begin{equation}\label{eq:wicommute}
[w^{(i)} , (\ev_0 \oplus \theta)(\mathcal C[0,1])] \subseteq \mathcal K(\mathbb C \oplus \mathcal H),\quad i=1,2,3
\end{equation}
and so that $w_1 := w^{(3)} w^{(2)} w^{(1)}$ satisfies \eqref{eq:ev1theta}. Then any product of unitary paths from $w^{(i)}$ to $\mathbf 1$ inside $C^\ast(w^{(i)})$ (which exist since each $w^{(i)}$ is self-adjoint), will produce a path $(w_t)_{t\in [0,1]}$ from $w_1$ to $\mathbf 1$ satisfying \eqref{eq:wtcommute}.

Let $X:= \{ (l,k) : l\in \mathbb N_0, k\in \{0,\dots, 2^l-1\}\}$ and decompose $X = X_1 \sqcup X_2 \sqcup X_3$ by $X_1 := \{(0,0)\}$,
\[
X_2 := \{ (l,k)\in X: k \textrm{ even, } l\neq 0\}, \quad X_3 := \{ (l,k)\in X: k \textrm{ odd}\}.
\]
Define (diagonal) representations $\pi, \rho \colon \mathcal C[0,1] \to \mathcal B(\ell^2(X))$ by
\[
\pi(f) \delta_{(l,k)} = f(k2^{-l})\delta_{(l,k)}, \quad \rho(f)\delta_{(l,k)} = f((k+1)2^{-l}) \delta_{(l,k)}
\]
for $f\in \mathcal C[0,1]$ and $(l,k)$, with subrepresentations $\pi_j, \rho_j \colon \mathcal C[0,1] \to \mathcal B(\ell^2(X_j))$ for $j=1,2,3$. Similarly, we define a representation $\eta_3 \colon \mathcal C[0,1] \to \mathcal B(\ell^2(X_3))$ by
\[
\eta_3(f) \delta_{(l,k)} = f((k-1)2^{-l}) \delta_{(l,k)}, \quad f\in \mathcal C[0,1],\, (l,k) \in X_3.
\]
It is easy to see that $(\pi, \rho)$, $(\pi_j, \rho_j)$ for $j=1,2,3$, and $(\pi_3, \eta_3)$ all define Cuntz pairs.
Define 
\[
\mathcal H:= \ell^2(X) \oplus \ell^2(X_3)\quad \textrm{and} \quad \theta := \pi \oplus \rho_3.
\]
 We will now produce the unitaries $w^{(1)}, w^{(2)},w^{(3)}$ as described above.

Let $U\in \mathcal U(\ell^2(X), \ell^2(X_3))$ be given by $U\delta_{(l,k)} = \delta_{(l+1,2k+1)}$. Then
\[
U\pi(f) \delta_{(l,k)} = f(k2^{-l}) \delta_{(l+1,2k+1)} =\eta_3(f)\delta_{(l+1,2k+1)} = \eta_3(f) U \delta_{(l,k)} 
\]
for $f\in \mathcal C[0,1]$ and $(l,k) \in X$, so $\ad U \circ \pi = \eta_3$. Similarly,
\[
U\rho(f) \delta_{(l,k)} =  f((k+1)2^{-l}) \delta_{(l+1,2k+1)} = \rho_3(f) \delta_{(l+1,2k+1)} = \rho_3(f) U \delta_{(l,k)}
\]
for $f\in \mathcal C[0,1]$ and $(l,k) \in X$, so $\ad U^\ast \circ \rho_3 = \rho$. Define
\[
w^{(1)} := \mathbf 1_{\mathbb C} \oplus \left( \begin{array}{cc} 0 & U^\ast \\ U & 0 \end{array} \right) \in \mathcal U(\mathbb C \oplus \ell^2(X) \oplus \ell^2(X_3))
\]
which is a self-adjoint unitary.
By the above computations we have
\[
\ad w^{(1)} \circ (\ev_0 \oplus \theta)  = \ad w^{(1)} \circ (\ev_0 \oplus \pi \oplus \rho_3) = \ev_0 \oplus \rho \oplus \eta_3. 
\]
Since $(\pi , \rho)$ and $(\rho_3, \eta_3)$ form Cuntz pairs, it follows that $\ad w^{(1)} \circ (\ev_0 \oplus \theta)$ and $\ev_0 \oplus \theta$ form a Cuntz pair, and therefore \eqref{eq:wicommute} is satisfied.

For $w^{(2)}$, we decompose 
\[
\mathbb C \oplus \mathcal H = \mathbb C \oplus \mathbb C \oplus \ell^2(X_2) \oplus \ell^2(X_3) \oplus \ell^2(X_3)
\]
where we used that $\ell^2(X_1) = \mathbb C$. Let $w^{(2)}$ be the (self-adjoint) unitary which swaps the two copies of $\mathbb C$ and of $\ell^2(X_3)$. Then
\[
\ad w^{(2)} \circ \ev_0 \oplus \rho \oplus \eta_3 = \rho_1 \oplus \ev_0 \oplus \rho_2 \oplus \eta_3 \oplus \rho_3.
\]
As above, one sees that $\ad w^{(2)}$ satisfies \eqref{eq:wicommute} (since everything that could form a Cuntz pair, does form a Cuntz pair). 

To define $w^{(3)}$ we let $V \in \mathcal U(\ell^2(X_2) , \ell^2(X_3))$ be given by
\[
V \delta_{(l,k)} = \delta_{(l,k+1)}, \qquad (l,k) \in X_2.
\]
Then for $f\in \mathcal C[0,1]$ and $(l,k) \in X_2$ we get
\[
V \rho_2(f) \delta_{(l,k)} = f((k+1)2^{-l}) \delta_{(l,k+1)} = \pi_3(f) \delta_{(l,k+1)} = \pi_3(f) V \delta_{(l,k)}, 
\]
and
\[
V\pi_2(f) \delta_{(l,k)} = f(k2^{-l}) \delta_{(l,k+1)} = \eta_3(f) \delta_{(l,k+1)} = \eta_3(f) V \delta_{(l,k)}.
\]
Hence $\ad V \circ \rho_2 = \pi_3$ and $\ad V^\ast \circ \eta_3 = \pi_2$. Define 
\[
w^{(3)} = \mathbf 1_{\mathbb C\oplus \mathbb C} \oplus \left( \begin{array}{cc} 0 & V^\ast \\ V & 0 \end{array} \right) \oplus \mathbf 1_{\ell^2(X_3)} \in \mathcal U(\mathbb C \oplus \mathbb C \oplus \ell^2(X_2) \oplus \ell^2(X_3) \oplus \ell^2(X_3))
\]
which is self-adjoint. Since $\ev_0 = \pi_1$ and $\rho_1 = \ev_1$, we get
\[
\ad w^{(3)} \circ (\rho_1 \oplus \ev_0 \oplus \rho_2 \oplus \eta_3 \oplus \rho_3) = \ev_1 \oplus \pi \oplus \rho_3 = \ev_1 \oplus \theta.
\]
As above, it follows that $\ad w^{(3)}$ satisfies \eqref{eq:wicommute}. In conclusion, with $w_1 := w^{(3)} w^{(2)} w^{(1)}$ we get
\[
\ad w_1 \circ (\ev_0 \oplus \theta) = \ev_1 \oplus \theta,
\]
and each $w^{(i)}$ is self-adjoint satisfying \eqref{eq:wicommute}, as desired.
\end{proof}

For the rest of the paper, our blanket assumption will be (unless specified otherwise) that $G$ is a second-countable locally compact group, that $A$ is a separable and $B$ a $\sigma$-unital \cstar-algebra, and that $\alpha: G\curvearrowright A$ and $\beta: G\curvearrowright B$ are continuous actions.


\section{From homotopy to stable operator homotopy}

\begin{nota} \label{nota:corona-notation}
We will denote the corona algebra of $B$ by $\CQ(B)=\CM(B)/B$.
We will not give the quotient map $\CM(B)\to\CQ(B)$ a name, but for an element $b\in\CM(B)$ or a \cstar-algebra $D\subseteq\CM(B)$, we will write $\bar{b}$ or $\bar{D}$ for the image under the quotient map.
Likewise, if $\phi: A\to\CM(B)$ is a map, we will write $\bar{\phi}$ for its composition with the quotient map.
If $\beta: G\curvearrowright B$ is an action, then $\bar{\beta}$ is the induced (algebraic) action on $\CQ(B)$.
\end{nota}

\begin{nota}
Given an action $\beta: G\curvearrowright B$ on a \cstar-algebra, denote
\[
\CM^\beta(B) = \{ x\in\CM(B) \mid \{x-\beta_g(x)\}_{g\in G}\subseteq B \}.
\]
Then $\CM^\beta(B)$ is a unital \cstar-subalgebra of $\CM(B)$ that contains the genuine fixed point subalgebra $\CM(B)^\beta$. 
In fact, under the quotient map $\CM(B)\to\CQ(B)$, it is the preimage of the \cstar-algebra $\CQ(B)^{\bar{\beta}}$, so there is an equivariant short exact sequence
\[
\xymatrix{
0 \ar[r] & B  \ar[r] & \CM^\beta(B)  \ar[r] & \CQ(B)^{\bar{\beta}} \ar[r] & 0.
}
\]
As a consequence, the restriction of $\beta$ to $\CM^\beta(B)$ is necessarily point-norm continuous; see \cite[Theorem 2.1]{Brown00}.
\end{nota}

\begin{nota}
For a cocycle representation $(\phi,\Iu): (A,\alpha)\to (\CM(B),\beta)$, denote
\[
D_{(\phi,\Iu)} = \CM^{\beta^\Iu}(B)\cap \{ x\in\CM(B) \mid  [x,\phi(A)]\subseteq B\}.
\]
Then $D_{(\phi,\Iu)}$ is a unital \cstar-algebra. 
In fact, under the quotient map $\CM(B)\to\CQ(B)$, $D_{(\phi,\Iu)}$ is the preimage of the \cstar-algebra $\bar{D}_{(\phi,\Iu)} = \big( \CQ(B)\cap\bar{\phi}(A)' \big)^{\bar{\beta}^\Iu}$, so there is a short exact sequence
\[
\xymatrix{
0 \ar[r] & B  \ar[r] & D_{(\phi,\Iu)}  \ar[r] & \big( \CQ(B)\cap\bar{\phi}(A)' \big)^{\bar{\beta}^\Iu}  \ar[r] & 0.
}
\]
\end{nota}

\begin{prop} \label{prop:D-unitaries}
Suppose $\beta$ is strongly stable and let $(\phi,\Iu): (A,\alpha)\to(\CM(B),\beta)$ be a cocycle representation.
A unitary $v\in\CU(\CM(B))$ belongs to $D_{(\phi,\Iu)}$ if and only if $(\phi,\Iu)$ forms an $(\alpha,\beta)$-Cuntz pair together with the cocycle representation $(\psi,\Iv)=\ad(v)\circ(\phi,\Iu)=(\ad(v)\circ\phi, v\Iu_\bullet\beta_\bullet(v)^*)$. 
\end{prop}
\begin{proof}
One always has
\[
\psi(a)-\phi(a)=[v,\phi(a)]v^*,\quad a\in A.
\]
So this difference is always in $B$ if and only if $[v,\phi(a)]\in B$ for all $a\in A$.
Furthermore, we have
\[
\Iv_g-\Iu_g = v\Iu_g\beta_g(v)^*-\Iu_g = (v-\beta^\Iu_g(v))\cdot \beta^\Iu_g(v)^*\Iu_g.
\]
So this difference is always in $B$ if and only if $\{v-\beta^\Iu_g(v)\}_{g\in G}\subseteq B$.
In conclusion, we see that $(\phi,\Iu)$ and $(\psi,\Iv)$ form an $(\alpha,\beta)$-Cuntz pair if and only if $v\in D_{(\phi,\Iu)}$.
\end{proof}

The following is a Cuntz--Thomsen picture analog of the better-known concept of (stable) operator homotopy in Kasparov's original approach to $KK$-theory; see for example \cite[Section 17.2]{BlaKK} or \cite[Definition 2.1.16]{JensenThomsen}.

\begin{defi} \label{def:operator-homotopy}
Suppose $\beta$ is strongly stable.
Let $(\phi,\Iu), (\psi,\Iv): (A,\alpha)\to(\CM(B),\beta)$ be two cocycle representations.
We say $(\phi,\Iu)$ and $(\psi,\Iv)$ are \emph{operator homotopic}, if there exists a unitary $u\in\CU_0(D_{(\phi,\Iu)})$ such that $(\psi,\Iv)=\ad(u)\circ(\phi,\Iu)$.
We call $(\phi,\Iu)$ and $(\psi,\Iv)$ \emph{stably operator homotopic}, if there exists some cocycle representation $(\kappa,\Ix): (A,\alpha)\to (\CM(B),\beta)$ such that $(\phi,\Iu)\oplus(\kappa,\Ix)$ and $(\psi,\Iv)\oplus(\kappa,\Ix)$ are operator homotopic.
\end{defi}

It is evident from \autoref{prop:D-unitaries} that if $(\phi,\Iu)$ is stably operator homotopic to $(\psi,\Iv)$, then they necessarily form an $(\alpha,\beta)$-Cuntz pair, which is homotopic to a degenerate Cuntz pair in the sense of \autoref{def:KKG-Thomsen}.
In particular, the class in $KK^G(\alpha,\beta)$ represented by this pair of cocycle representations must vanish.
In the spirit of similar results in the literature focusing on more special cases, the main achievement in this section is that the converse also holds.
That is, for any anchored $(\alpha,\beta)$-Cuntz pair, its associated class in $KK^G(\alpha,\beta)$ is trivial precisely when the pair of cocycle representations is stably operator homotopic.

We recall Kasparov's technical theorem \cite[Theorem 1.4]{Kasparov88}.
It is worth pointing out that the proof given there (due to Higson \cite{Higson87_2}) is not as bad as the name suggests; it arises as a somewhat elaborate application of \autoref{lem:Kasparov} whose proof in Kasparov's original paper fits on two pages.
The way it is stated here is a slight reformulation compared to how it is stated in the reference.
There are two differences we want to point out.
Firstly, the element $M$ in our statement is the squareroot of the element $M_1$ and $N$ is the squareroot of $M_2$ in \cite[Theorem 1.4]{Kasparov88}.
Secondly, the map $\If$ below (denoted as $\phi$ in the original version) is allowed to have a more general domain as opposed to just the acting group $G$.
This is motivated by the fact that the only property about the domain ever used in the original proof of the theorem is that it is $\sigma$-compact.

\begin{theorem} \label{thm:Kasparov}
Let $G$ be a second-countable, locally compact group.
Let $B$ be a $\sigma$-unital \cstar-algebra and $\beta: G\curvearrowright B$ a continuous action.
Let $E_1, E_2\subset\CM(B)$ two $\sigma$-unital \cstar-subalgebras such that $E_1\cdot E_2\subseteq B$, and $E_1$ is $\beta$-invariant with $\beta|_{E_1}$ point-norm continuous.
Let $X$ be a locally compact, $\sigma$-compact Hausdorff space.
Suppose that $\If: X\to\CM(B)$ is a bounded strictly continuous map such that for all $a\in E_1$, the maps $a\cdot\If$ and $\If\cdot a$ take values in $B$ and are norm-continuous on $X$.
Let furthermore $\Delta\subseteq\CM(B)$ be a norm-separable subset with $[\Delta, E_1]\subseteq E_1$.
Then there exist two positive contractions $N,M\in\CM^\beta(B)$ with $N^2+M^2=\eins$ and satisfying
\[
ME_1\subseteq B,\ NE_2\subseteq B,\ [M,\Delta],\ [N,\Delta]\subseteq B,
\]
and such that the functions $N\cdot\If$ and $\If\cdot N$ define norm-continuous maps on $X$ with values in $B$.
\end{theorem}

\begin{nota} \label{nota:op-homotopy}
The given representation and unitary path from \autoref{thm:basic-op-homotopy} will play an important role for the rest of this section, in the sense as we are about to specify now.
For a \cstar-algebra $B$, we will for $t\in [0,1]$ also (by slight abuse of notation) denote by $\ev_t: B[0,1]\to B$ the obvious evaluation map.
For a given separable infinite-dimensional Hilbert space $\CH$, we keep in mind that $\CB(\CH)=\CM(\CK(\CH))$.
Given an action $\beta: G\curvearrowright B$, we may tensor the representation from \autoref{thm:basic-op-homotopy} with the identity map on $B$ and the corresponding unitary path with the unit of $\CM(B)$ to get the equivariant representation
\[
\theta\otimes\id_B: (B[0,1],\beta[0,1])\to (\CM(\CK(\CH)\otimes B),\id\otimes\beta)
\] 
and a norm-continuous unitary path 
\[
w\otimes\eins: [0,1]\to\CU(\CM(\CK(\IC\oplus\CH)\otimes B))^{\id\otimes\beta}
\]
with $w_0=\eins$ such that obviously
\[
(\ev_1\oplus\theta)\otimes\id_B = \ad(w_1\otimes\eins)\circ\big( (\ev_0\oplus\theta)\otimes\id_B \big)
\]
and
\[
\big[ w_t\otimes\eins , \big( (\ev_0\oplus\theta)\otimes\id_B \big)(B[0,1]) \big] \subseteq \CK(\IC\oplus\CH)\otimes B,\quad t\in [0,1].
\]
Now let us additionally assume that $\beta$ is strongly stable.
Choose a sequence $r_n\in\CM(B)^\beta$ of isometries for $n\geq 0$ such that $\eins=\sum_{n=0}^\infty r_n r_n^*$ in the strict topology.
Let us also define the isometry $r_\infty=\sum_{k=0}^\infty r_{k+1}r_k^*\in\CM(B)^\beta$, which then fits into the equation $r_0r_0^*+r_\infty r_\infty^*=\eins$.
Let $\set{e_{k,l} \mid k,l\geq 0}\subset\CK(\IC\oplus\CH)$ be a set of generating matrix units such that $e_{0,0}$ is the orthgonal projection onto $\IC\oplus 0$.
Then $\set{e_{k,l}\mid k,l\geq 1}$ necessarily generates the \cstar-subalgebra $\CK(\CH)$.
We have two equivariant isomorphisms
\[
\Lambda_0 : \CK(\IC\oplus\CH)\otimes B\to B,\quad \Lambda_1: \CK(\CH)\otimes B\to B
\]
determined by the formulas 
\[
\Lambda_0(e_{k,l}\otimes b)=r_k b r_l^* \quad (k,l\geq 0),\text{ and } \Lambda_1(e_{k,l}\otimes b)=r_{k-1} b r_{l-1}^* \quad (k,l\geq 1).
\]
One has the identity $r_\infty\Lambda_1(e_{k,l}\otimes b)r_\infty^*=\Lambda_0(e_{k,l}\otimes b)$ for all $k,l\geq 1$, which implies $r_\infty\Lambda_1(\_)r_\infty^*=\Lambda_0(0\oplus\id_{\CK(\CH)\otimes B})$.
We extend $\Lambda_0$ and $\Lambda_1$ to equivariant isomorphisms between the multiplier algebras as well.

We consider the non-degenerate equivariant $*$-homomorphism 
\[
\theta^B=\Lambda_1\circ(\theta\otimes\id_B): B[0,1] \to \CM(B)
\] 
and the norm-continuous unitary path 
\[
w^B_t=\Lambda_0(w_t\otimes\eins)\in\CM(B)^\beta,\quad t\in [0,1].
\]
We then observe for all $f\in B[0,1]$
\[
\begin{array}{cl}
\multicolumn{2}{l}{ \Lambda_0\circ\big( (\ev_t\oplus\theta)\otimes\id_B \big)(f) }\\
=& \dst \Lambda_0(e_{0,0}\otimes f(t)) + \Lambda_0\big( 0\oplus(\theta\otimes\id_B)(f) \big)  \\
=& \dst r_0 f(t) r_0^* + r_\infty\Lambda_1((\theta\otimes\id_B)(f))r_\infty^*  \\
=& (\ev_t\oplus_{r_0,r_\infty} \theta^B)(f).
\end{array}
\]
In summary, we have used \autoref{thm:basic-op-homotopy} and the strong stability of $\beta$ to construct a non-degenerate equivariant $*$-homomorphism $\theta^B: (B[0,1],\beta[0,1])\to (\CM(B),\beta)$ and a norm-continuous unitary path $w^B: [0,1]\to\CU(\CM(B)^\beta)$ with $w^B_0=\eins$ and such that
\[
\ev_1\oplus_{r_0,r_\infty} \theta^B = \ad(w_1^B)\circ(\ev_0\oplus_{r_0,r_\infty} \theta^B)
\]
and
\[
\Big[ w^B_t, (\ev_0\oplus_{r_0,r_\infty} \theta^B)(B[0,1]) \Big] \subseteq B,\quad t\in [0,1].
\]
This witnesses the fact that the endpoint evaluation maps $\ev_0,\ev_1: B[0,1]\to B$ are stably operator homotopic.
The construction outlined above depends on the choice of the isometries $r_j$, but only up to equivalence with a unitary in $\CU(\CM(B)^\beta)$.
We will subsequently abuse notation and simply write $w$ in place of $w^B$.
We will also write $\CE_i=\ev_i\oplus_{r_0,r_\infty}\theta^B$ for $i=0,1$.
Each of $\ev_i$, $\theta^B$ and $\CE_i$ is a non-degenerate $\beta$-equivariant $*$-homomorphism from $B[0,1]$ to $\CM(B)$.
In particular they extend uniquely to unital equivariant $*$-homomorphisms from $\CM(B[0,1])$ to $\CM(B)$ that are strictly continuous on bounded sets.
We also denote their extensions by $\ev_i$, $\theta^B$ and $\CE_i$.
\end{nota}

The following technical lemma is the key ingredient in the main result of this section, \autoref{thm:homotopy-implies-operator-homotopy}.
Its punny name is a not so subtle nod to the circumstances related to the lemma's discovery.
We record the statement in a somewhat more general and explicit form than we need it for the rest of the paper, keeping in mind potential further applications.
For example, the lemma below could easily be used to not only prove the principle ``homotopy implies stable operator homotopy'', but also its appropriate analogs in other variants of the (equivariant) $KK$-groups, such as the nuclear or ideal-related versions.

\begin{lemma}[KKarantine lemma] \label{lem:homotopy-implies-operator-homotopy}
Suppose that $\beta$ is strongly stable.
Let $\big( (\Phi,\IU), (\Psi,\IV) \big)$ be an $(\alpha, \beta[0,1])$-Cuntz pair such that 
\[
(\phi,\Iu) :=\ev_0\circ(\Phi,\IU)=\ev_1\circ(\Phi,\IU)=\ev_0\circ(\Psi,\IV).
\]
Define $\ (\psi,\Iv):=\ev_1\circ(\Psi,\IV)$ and let $\mathcal E_0$ and $\theta^B$ be as in \autoref{nota:op-homotopy}. Then
\[
(\kappa,\Ix) :=\theta^B\circ(\Psi,\IV)\oplus\CE_0\circ(\Phi,\IU) \colon (A,\alpha) \to (\CM(B),\beta)
\]
is a cocycle representation, and $(\phi, \Iu) \oplus (\kappa, \Ix)$ and $(\psi,\Iv) \oplus (\kappa, \Ix)$ are operator homotopic.
\end{lemma}
\begin{proof}
That $(\kappa, \Ix)$ is a cocycle representation follows since $\theta^B$ and $\CE_0$ are unital, equivariant and strictly continuous on bounded sets.

For the rest of the proof, we shall use the notation $x\equiv_B y$ for two multipliers $x,y\in\CM(B)$ to mean that $x-y\in B$.
We adopt the choices and notation from the last paragraph in \autoref{nota:op-homotopy}.
It shall be understood that Cuntz sums of elements are formed with the pair of isometries $r_0$ and $r_\infty$ as defined there, unless we specify otherwise. 
Then $\Xi: M_2(\CM(B))\to \CM(B)$ given by
\[
\Xi\matrix{ a & b \\ c & d } = r_0ar_0^* + r_0 b r_\infty^* + r_\infty c r_0^* + r_\infty d r_\infty^*
\]
is an (equivariant) isomorphism with $\Xi(M_2(B))=B$.
In particular one has $b_1\oplus b_2=\Xi(\diag(b_1,b_2))$ for all $b_1,b_2\in\CM(B)$.
Define
\[
(\theta, \Iy) := (\phi,\Iu)\oplus(\kappa,\Ix) =\CE_0\circ(\Psi,\IV) \oplus \CE_0\circ(\Phi,\IU)
\]
and note that
\[
(\psi,\Iv)\oplus(\kappa,\Ix)=\CE_1\circ(\Psi,\IV) \oplus \CE_1\circ(\Phi,\IU) = \ad(w_1 \oplus w_1^*)\circ (\theta,\Iy).
\]
(Strictly speaking, the left-most Cuntz addition ``$\oplus$'' appearing in these particular instances here has to be performed with different isometries than the pair $r_0, r_\infty$ chosen before to achieve these equations.)
Here we use $\ad(w_1)\circ\CE_0=\CE_1$ and $\CE_0\circ(\Phi,\IU)=\CE_1\circ(\Phi,\IU)$.
We claim that the unitary $w_1\oplus w_1^*$ is homotopic to the unit within $D_{(\theta,\Iy)}$, which will complete the proof.
We shall now construct such a homotopy.

We first observe that since $\CE_0$ is equivariant, unital and strictly continuous on bounded sets, the map $\CE_0\circ\IV: G\to\CU(\CM(B))$ is a strictly continuous $\beta$-cocycle.
Define a subset of $\CM(B)$ via
\[
\Delta=\CE_0\circ\Psi(A)\cup\set{w_t}_{t\in [0,1]},
\]
which is norm-separable since $A$ is separable and $w$ is norm-continuous. 
Define \cstar-subalgebras of $\CM(B)$ via
\[
E_1=\CE_0(B[0,1])+B
\]
and
\[
E_2=\cstar\Big( w_t(\CE_0\circ\Psi)(a)w_t^*-(\CE_0\circ\Psi)(a)  \mid a\in A,\ t\in [0,1] \Big).
\]
Furthermore, we define a map
\[
\If: [0,1]\times G\to\CM(B),\quad\If(t,g)=\CE_0(\IV_g)w_t\CE_0(\IV_g)^*-w_t,
\]
which is strictly continuous since multiplication is strictly continuous on bounded sets. 
We clearly have that $E_2$ is separable and $E_1$ is $\sigma$-unital.
$E_1$ is also invariant under the cocycle perturbed action $\beta^{\CE_0(\IV)}$ and $\beta^{\CE_0(\IV)}|_{E_1}$ is point-norm continuous.
By the choice of the unitaries $w_t$, they commute with elements of $E_1$ modulo $B$, hence we see that $[\Delta,E_1]\subseteq E_1$.
It also follows for arbitrary elements $z\in\CM(B[0,1])$ and $b\in B[0,1]$ that
\[
(w_t\CE_0(z)w_t^*-\CE_0(z))\CE_0(b) = w_t\CE_0(zb)w_t^*-\CE_0(zb)+w_t\CE_0(z)[w_t^*,\CE_0(b)] \ \in \ B.
\]
If $z$ is additionally assumed to be a unitary, then further
\[
\begin{array}{cl}
\multicolumn{2}{l}{ (\CE_0(z)w_t\CE_0(z)^*-w_t)\CE_0(b) }\\ 
=& \CE_0(z)(w_t\CE_0(z)^*w_t^*-\CE_0(z)^*) w_t\CE_0(b) \\
=& \CE_0(z)(w_t\CE_0(z)^*w_t^*-\CE_0(z)^*)[w_t,\CE_0(b)] \\
& +\CE_0(z)(w_t\CE_0(z)^*w_t^*-\CE_0(z)^*)\CE_0(b)w_t \ \in \ B.
\end{array}
\]
Analogously one has also $\CE_0(b)(\CE_0(z)w_t\CE_0(z)^*-w_t)\in B$.
These observations immediately imply
\[
E_1\cdot E_2 \subseteq B \quad\text{and}\quad E_1\cdot\If([0,1]\times G)\cup\If([0,1]\times G)\cdot E_1\subseteq B.
\]
Because the $\beta[0,1]$-cocycle $\IV$ is strictly continuous and the path $w$ was norm-continuous, the maps of the form $e\cdot\If$ and $\If\cdot e$ are norm-continuous on $[0,1]\times G$ for all $e\in E_1$. In fact, this is an easy consequence of the fact (which is elementary to prove), that if $X$ is a topological space, $f: X \to B$ is norm-continuous and $g: X \to \CM(B)$ is stricly continuous and bounded, then the maps $x\mapsto f(x)g(x)$ and $x\mapsto g(x) f(x)$ are norm-continuous. 

In particular, this allows us to apply Kasparov's technical \autoref{thm:Kasparov} to this setup, with the cocycle perturbed $G$-action $\beta^{\CE_0(\IV)}$ on $B$.
We find positive contractions $N,M\in\CM^{\beta^{\CE_0(\IV)}}(B)$ with $N^2+M^2=\eins$ such that
\begin{equation} \label{eq:op-homo:NM}
ME_1\subseteq B,\ NE_2\subseteq B,\ [M,\Delta],\ [N,\Delta]\subseteq B,
\end{equation}
and moreover the maps $N\cdot\If$ and $\If\cdot N$ are norm-continuous maps on $[0,1]\times G$ with values in $B$.
Note that $M$ and $N$ commute, and that $[M,E_2] \subseteq B$ and $[N,E_1] \subseteq B$. 
We consider the unitary
\[
U=\Xi\matrix{ N & M \\ -M & N } \in \CU(\CM(B)).
\]
A trivial computation shows for all $x,y\in\CM(B)$ that
\begin{equation} \label{eq:op-homo:conjugate-U}
U\Xi\matrix{ x & 0 \\ 0 & y }U^* = \Xi\matrix{ NxN+MyM & MyN-NxM \\ NyM-MxN & MxM+NyN }.
\end{equation}
If one has $y\in\Delta$ and $y-x\in E_2$, then condition \eqref{eq:op-homo:NM} implies that
\begin{equation} \label{eq:op-homo:flip1}
U\Xi\matrix{ x & 0 \\ 0 & y }U^* \equiv_B \Xi\matrix{ y & 0 \\ 0 & x }.
\end{equation}
In particular it follows for any $a\in A$ and any $t\in [0,1]$ that the pair of elements $x=\ad(w_t)\circ\CE_0\circ\Psi(a)$ and $y=\CE_0\circ\Psi(a)$ satisfies this property.
Furthermore, since $N$ acts like a unit on $E_1$ modulo $B$, we have that $U$ acts like a unit on $\Xi(M_2(E_1))$ modulo $B$.
Keeping in mind that $\CE_0\circ(\Phi-\Psi)(A)\subseteq E_1$, we compute for all $a\in A$ and all $t\in [0,1]$ that one has
\[
\begin{array}{cl}
\multicolumn{2}{l}{ U \Xi\matrix{ \ad(w_t)\circ\CE_0\circ\Psi(a) & 0 \\ 0 & \CE_0\circ\Phi(a) } U^* }\\
\equiv_{\makebox[0pt]{\footnotesize\hspace{2mm}$B$}}& \Xi\matrix{ 0 & 0 \\ 0 & \CE_0\circ(\Phi-\Psi)(a) } + U \Xi\matrix{ \ad(w_t)\circ\CE_0\circ\Psi(a) & 0 \\ 0 & \CE_0\circ\Psi(a) } U^* \\
\stackrel{\eqref{eq:op-homo:flip1}}{\equiv}_{\makebox[0pt]{\footnotesize $B$}}& \Xi\matrix{ \CE_0\circ\Psi(a) & 0 \\ 0 & [\ad(w_t)-\id]\circ\CE_0\circ\Psi(a) + \CE_0\circ\Phi(a) }
\end{array}
\]
By the choice of the unitary path $w$, we have that $[\ad(w_t)-\id]\circ\CE_0(B[0,1])\subseteq B$.
Since $\Phi(a)$ and $\Psi(a)$ agree modulo $B[0,1]$, we have that 
\[
[\ad(w_t)-\id]\circ\CE_0\circ(\Psi-\Phi)(a)\in B,
\] 
and therefore, the last expression above agrees modulo $B$ with
\[
\Xi\matrix{ \CE_0\circ\Psi(a) & 0 \\ 0 & \ad(w_t)\circ\CE_0\circ\Phi(a) }.
\]
Since $a\in A$ was arbitrary, we get for all $t\in [0,1]$ that the unitary 
\[
(\eins\oplus w_t^*)U(w_t\oplus\eins) 
\] 
commutes with the range of the map $\theta=\CE_0\circ\Psi\oplus\CE_0\circ\Phi$ modulo $B$.

Now let $t\in [0,1]$ be given. 
We want to show that the unitary $(\eins\oplus w_t^*)U(w_t\oplus\eins)$ is invariant under the cocycle perturbed action $\beta^\Iy$ modulo $B$, where $\Iy=\CE_0\circ\IV\oplus\CE_0\circ\IU$.
Since $w_t\in\Delta$ commutes with both $N$ and $M$ modulo $B$, we can see that 
\begin{equation} \label{eq:op-homo:simplify-unitary}
(\eins\oplus w_t^*)U(w_t\oplus\eins)=\Xi\matrix{ Nw_t & M \\ -w_t^*M w_t & w_t^* N } \equiv_B \Xi\matrix{ Nw_t & M \\ -M & w_t^* N }.
\end{equation}
Recall that the elements $M, N$ are fixed by the action $\beta^{\CE_0(\IV)}$ modulo $B$.
Since $\IU$ and $\IV$ are assumed to agree modulo $B[0,1]$, the map $g\mapsto\IV_g\IU_g^*=\eins+(\IV_g-\IU_g)\IU_g^*$ takes values in $\eins+B[0,1]$ and hence $\CE_0(\IV_g\IU_g^*)\in (\eins+E_1)$.
Using furthermore that $w_t$ is fixed by $\beta$, we compute
\[
\renewcommand\arraystretch{1.5}
\begin{array}{cl}
\multicolumn{2}{l}{ \beta_g^\Iy\Big( (\eins\oplus w_t^*)U(w_t\oplus\eins) \Big)}\\
\stackrel{\eqref{eq:op-homo:simplify-unitary}}{\equiv}_{\makebox[0pt]{\footnotesize $B$}} & \Xi\matrix{ \CE_0(\IV_g) \beta_g(Nw_t) \CE_0(\IV_g)^* & \CE_0(\IV_g)\beta_g(M)\CE_0(\IU_g)^* \\ -\CE_0(\IU_g)\beta_g(M)\CE_0(\IV_g)^* & \CE_0(\IU_g)\beta_g(w_t^* N)\CE_0(\IU_g)^* } \\
\equiv_{\makebox[0pt]{\footnotesize\hspace{2mm}$B$}} & \Xi\matrix{ N\CE_0(\IV_g) w_t \CE_0(\IV_g)^* & M\CE_0(\IV_g\IU_g^*) \\ -\CE_0(\IU_g\IV_g^*)M & \CE_0(\IU_g)w_t^*\CE_0(\IV_g)^* N \CE_0(\IV_g\IU_g^*) } \\
\stackrel{\eqref{eq:op-homo:NM}}{\equiv}_{\makebox[0pt]{\footnotesize $B$}} & 
\Xi\matrix{ N\CE_0(\IV_g) w_t \CE_0(\IV_g)^* & M \\ -M & \CE_0(\IU_g)w_t^*\CE_0(\IV_g)^* N \CE_0(\IV_g\IU_g^*) } \\
= & \Xi\matrix{ N\CE_0(\IV_g) w_t \CE_0(\IV_g)^* & M \\ -M & \CE_0(\IU_g\IV_g^*) \big( \CE_0(\IV_g)w_t^*\CE_0(\IV_g)^* N \big) \CE_0(\IV_g\IU_g^*) }.
\end{array}
\]
We note that
\[
N\CE_0(\IV_g) w_t \CE_0(\IV_g)^*-Nw_t = N\If(t,g)\in B.
\]
Therefore
\[
\beta_g^\Iy\Big( (\eins\oplus w_t^*)U(w_t\oplus\eins) \Big) \equiv_B \Xi\matrix{ Nw_t & M \\ -M & \CE_0(\IU_g\IV_g^*) w_t^* N \CE_0(\IV_g\IU_g^*) }.
\]
Using again that $\CE_0(\IV_g\IU_g^*)\in (\eins+E_1)$, we have that both $w_t$ and $N$ commute with such an element modulo $B$, which leads to
\[
\renewcommand\arraystretch{1.5}
\begin{array}{ccl}
\beta_g^\Iy\Big( (\eins\oplus w_t^*)U(w_t\oplus\eins) \Big) 
& \equiv_{\makebox[0pt]{\footnotesize\hspace{2mm}$B$}} &
\Xi\matrix{ Nw_t & M \\ -M & w_t^* N } \\
& \stackrel{\eqref{eq:op-homo:simplify-unitary}}{\equiv}_{\makebox[0pt]{\footnotesize $B$}} &
(\eins\oplus w_t^*)U(w_t\oplus\eins).
\end{array}
\]
Due to all the properties we have verified for $U$ so far, we may finally conclude that
\[
(\eins\oplus w_t^*)U(w_t\oplus\eins) \in D_{(\theta,\Iy)},\quad t\in [0,1].
\]
Since $w_0=\eins$, this is in particular the case for $U$ itself. 
Moreover, since $\ad w_1 \circ \CE_0 = \CE_1$ and $\CE_0\circ (\Phi, \IU) = \CE_1\circ (\Phi, \IU) = \CE_0\circ (\Psi, \IV)$, it is immediate to check that $w_1\oplus\eins$ and $\eins\oplus w_1$ are elements of $D_{(\theta,\Iy)}$.

We can readily observe that $U+U^*\geq 0$, so the spectrum of $U$ is contained in the right-half circle.
Let us denote by $\log: \IC\setminus\IR^{-}\to\IC$ the standard holomorphic branch of the logarithm.
We then define a norm-continuous unitary path $V: [0,3]\to \CU(D_{(\theta,\Iy)})$ via
\[
V_t=\begin{cases} \exp(t\log(U)) &,\quad 0\leq t\leq 1 \\
(\eins\oplus w_{t-1}^*)U(w_{t-1}\oplus\eins) &,\quad 1\leq t\leq 2 \\
(\eins\oplus w_1^*) \exp( (3-t)\log(U) ) (w_1\oplus\eins) &,\quad 2\leq t\leq 3. \end{cases}
\]
Evidently we have $V_0=\eins$ and $V_3=w_1\oplus w_1^*$.
In summary, the cocycle representations $(\phi,\Iu)$ and $(\psi,\Iv)$ are indeed stably operator homotopic and the proof is complete.
\end{proof}

\begin{theorem} \label{thm:homotopy-implies-operator-homotopy}
Suppose $\beta$ is strongly stable.
Let $(\phi,\Iu), (\psi,\Iv): (A,\alpha) \to (\CM(B),\beta)$ be two cocycle representations that form an anchored $(\alpha,\beta)$-Cuntz pair.
If $[(\phi,\Iu), (\psi,\Iv)] = 0 \in KK^G(\alpha,\beta)$, then $(\phi,\Iu)$ and $(\psi,\Iv)$ are stably operator homotopic.
\end{theorem}
\begin{proof}
Since $\big( (\phi,\Iu), (\psi,\Iv)\big), \big( (\phi,\Iu), (\phi,\Iu) \big) \in \IE^G(\alpha, \beta)$ are anchored and vanish in $KK^G(\alpha, \beta)$, it follows from \autoref{prop:KKG-homotopy} that they are homotopic.
So there is an $(\alpha, \beta[0,1])$-Cuntz pair $\big( (\Phi,\IU), (\Psi,\IV) \big)$ such that 
\[
(\phi,\Iu)=\ev_0\circ(\Phi,\IU)=\ev_1\circ(\Phi,\IU)=\ev_0\circ(\Psi,\IV) \text{ and } (\psi,\Iv)=\ev_1\circ(\Psi,\IV).
\]
\autoref{lem:homotopy-implies-operator-homotopy} implies that $(\phi, \Iu)$ and $(\psi,\Iv)$ are stably operator homotopic.
\end{proof}


\section{Absorption of cocycle representations}
 
In this section, we consider various more notions of equivalence and subequivalence between cocycle representations $(A,\alpha)\to(\CM(B),\beta)$, some of which we are about to introduce.
These are similar (but not identical) to the notions of equivalence considered in \cite[Section 2]{Szabo21cc}.
Compared to the original work of Dadarlat--Eilers, the content of this section should be seen as a dynamical generalization of a number of technical lemmas in \cite[Subsection 2.3]{DadarlatEilers02} and those in some references invoked therein.
There are a lot of similarities to some results and their proofs in \cite{Thomsen05}, with the important difference that said reference only deals with genuinely equivariant representations.
The following is an analog of \cite[Definition 2.1]{DadarlatEilers01}.

\begin{defi} \label{def:various-equivalences}
Let $(\phi,\Iu), (\psi,\Iv): (A,\alpha)\to (\CM(B),\beta)$ be two cocycle representations.
We write $(\phi,\Iu)\asymp (\psi,\Iv)$, if there exists a norm-continuous path $u: [0,\infty)\to\CU(\CM(B))$ such that
	\begin{itemize}
  	\item $\psi(a) = \dst\lim_{t\to\infty} u_t\phi(a)u_t^*$ for all $a\in A$;
  	\item $\psi(a)-u_t\phi(a)u_t^* \in B$ for all $a\in A$ and $t\geq 0$;  
	\item $\dst\lim_{t\to\infty} \max_{g\in K}\ \| \Iv_g-u_t\Iu_g\beta_g(u_t)^* \| =0$ for all compact sets $K\subseteq G$;
  	\item The map $[(t,g)\mapsto \Iv_g-u_t\Iu_g\beta_g(u_t)^*]$ takes values in $B$.
  \end{itemize}
\end{defi}
In view of \autoref{prop:D-unitaries}, if $\big((\phi, \Iu), (\psi,\Iv)\big)$ is a Cuntz pair, then the second and fourth bullet points above are equivalent to saying that the range of $u$ is in the \cstar-algebra $D_{(\phi,\Iu)}$. 

Similarly to how we formed Cuntz sums of two elements, we may form infinite sums by a similar method if the underlying action is strongly stable.

\begin{defi} \label{def:inf-repeat}
Suppose that $\beta$ is strongly stable.
Let $t_n\in \CM(B)^\beta$ be any sequence of isometries such that $\sum_{n=1}^\infty t_nt_n^* = \eins$ in the strict topology.
Then we have a $\beta$-equivariant $*$-homomorphism
\[
\ell^\infty(\IN, \CM(B)) \to \CM(B),\quad (b_n)_{n\geq 1} \mapsto \sum_{n=1}^\infty t_nb_nt_n^*,
\]
which is jointly strictly continuous on the unit ball of $\ell^\infty(\IN, \CM(B))$ and does not depend on the choice of $t_n$ up to unitary equivalence with a unitary in $\CM(B)^\beta$.(Similarly as before, if $v_n\in\CM(B)^\beta$ is another sequence of isometries satisfying the same relation, then the unitary defined as the strict limit $w=\sum_{n=1}^\infty t_nv_n^*$ implements this equivalence.)
For any sequence of cocycle representations $(\phi^{(n)},\Iu^{(n)}): (A,\alpha)\to (\CM(B),\beta)$, we may hence define the countable sum
\[
(\Phi,\IU)=\bigoplus_{n=1}^\infty (\phi^{(n)},\Iu^{(n)}): (A,\alpha)\to (\CM(B),\beta)
\]
via the pointwise strict limits
\[
\Phi(a)=\sum_{n=1}^\infty t_n\phi^{(n)}(a)t_n^*,\quad \IU_g = \sum_{n=1}^\infty t_n\Iu^{(n)}_g t_n^*.
\]
Up to equivalence with a unitary in $\CM(B)^\beta$, this cocycle representation does not depend on the choice of $(t_n)_n$.
In particular, in the special case that $(\phi^{(n)},\Iu^{(n)})=(\phi,\Iu)$ for all $n$, we denote the resulting countable sum by $(\phi^\infty, \Iu^\infty)$ and call it the \emph{infinite repeat} of $(\phi,\Iu)$.
\end{defi}

\begin{defi} \label{def:wc}
Let $(\psi,\Iv): (A,\alpha) \to (\CM(B),\beta)$ be a cocycle representation, and let $\FC$ be a family of cocycle representations from $(A,\alpha)$ to $(\CM(B),\beta)$.
We say that $(\psi,\Iv)$ is \emph{weakly contained} in $\FC$, written $(\psi,\Iv)\wc \FC$,
if the following is true.

For all compact sets $K\subseteq G$, $\CF\subset A$, every $\eps\greater 0$, and every contraction $b\in B$ there exist $(\phi^{(1)},\Iu^{(1)}),\dots,(\phi^{(\ell)},\Iu^{(\ell)})\in\FC$ and a collection of elements
\[
\{c_{j,k} \mid j=1,\dots, N,\ k=1,\dots,\ell\}\subset B
\]
satisfying
\begin{equation} \label{eq:wc:1}
\max_{g\in K} \Big\| b^*\Iv_g\beta_g(b)-\sum_{k=1}^\ell \sum_{j=1}^{N} c_{j,k}^*\Iu^{(k)}_g\beta_g(c_{j,k}) \Big\| \leq \eps
\end{equation}
and
\begin{equation} \label{eq:wc:2}
\max_{a\in\CF} \Big\| b^*\psi(a)b - \sum_{k=1}^\ell \sum_{j=1}^N c_{j,k}^*\phi^{(k)}(a)c_{j,k} \Big\| \leq \eps.
\end{equation}
In particular, if $\FC$ consists of a single cocycle representation $(\phi,\Iu)$, then we write $(\psi,\Iv)\wc (\phi,\Iu)$,
in place of $(\psi,\Iv)\wc\set{(\phi,\Iu)}$, and in this case one can choose $\ell=1$ above.
\end{defi}

\begin{rem}
It is straightforward to verify that $\wc$ is a reflexive and transitive relation on the set of all cocycle representations.

When $G=\set{1}$ and $A,\phi,\psi$ are all unital, then the above recovers ``approximate domination'' as u.c.p.\ maps, as for example considered in \cite[Definition 3.1]{Gabe20}.
On the other hand, when $A=\IC$, $B=\CK$, $\beta=\id_\CK$ and $\phi,\psi$ are unital, then the above recovers weak containment of unitary representations $G\to\CU(\CH)\subset\CM(\CK)$, where $\CH$ is an infinite-dimensional separable Hilbert space.
This is a consequence of comparing \autoref{cor:coc-rep-absorption} below with Voiculescu's theorem \cite{Voiculescu76} and its known variant for unitary representations.
\end{rem}

The following is a dynamical analog of \cite[Definition 2.11]{DadarlatEilers02}.

\begin{defi} \label{def:contained-at-infty}
Let $(\phi,\Iu), (\psi,\Iv): (A,\alpha) \to (\CM(B),\beta)$ be two cocycle representations.
We say that $(\psi,\Iv)$ is \emph{contained in $(\phi,\Iu)$ at infinity}, if the following is true.
For all compact sets $K\subseteq G$, $\CF\subset A$, every $\eps> 0$, and every contraction $b\in B$ there exists an element $x\in B$ satisfying
\[
\max_{g\in K} \| b^*\Iv_g\beta_g(b)-x^*\Iu_g\beta_g(x) \| \leq \eps,
\]
\[
\max_{a\in\CF} \| b^*\psi(a)b - x^*\phi(a)x \| \leq \eps,
\]
and
\[
\| x^*b \| \leq \eps.
\]
\end{defi}

If we compare \autoref{def:contained-at-infty} with \autoref{def:wc}, then we can readily see that if $(\psi,\Iv)$ is contained in $(\phi,\Iu)$ at infinity, then $(\psi,\Iv)$ is weakly contained in $(\phi,\Iu)$.

\begin{rem}
Note that by considering the case where $1\in K$ in \eqref{eq:wc:1} we obtain $\| b^\ast b - \sum_{j=1}^\ell \sum_{k=1}^N c_{j,k}^\ast c_{j,k}\| \leq \eps$. 
Consequently, it easily follows that $(\psi, \Iv) \wc (\phi, \Iu)$ if and only if $(\psi^\dagger, \Iv) \wc (\phi^\dagger, \Iu)$, where 
\[
(\phi^\dagger,\Iu), (\psi^\dagger,\Iv) \colon (A^\dagger, \alpha^\dagger) \to (\CM(B), \beta)
\]
 are the induced unital cocycle representation out of the proper unitization. Similarly, $(\psi, \Iv)$ is contained in $(\phi,\Iu)$ at infinity if and only if $(\psi^\dagger, \Iv)$ is contained in $(\phi^\dagger, \Iu)$ at infinity.
This is in stark contrast to the non-equivariant case where this subtlety has led to false statements in the literature, such as the one addressed by the first named author in \cite{Gabe16}.
\end{rem}

We make a few basic observations about weak containment.

\begin{prop} \label{prop:wc-observations}
Let 
\[
(\phi,\Iu), (\psi,\Iv), (\psi^{(n)},\Iv^{(n)}): (A,\alpha)\to (\CM(B),\beta)
\] 
be cocycle representations for all $n \in \mathbb N$.
\begin{enumerate}[label=\textup{(\roman*)},leftmargin=*]
\item \label{prop:wc-observations:1}
If we have for all $a\in A$, $b\in B$ and compact sets $K\subseteq G$ that
\[
\lim_{n\to\infty} \Big( \| \big( \psi(a)-\psi^{(n)}(a) \big) b \| + \max_{g\in K} \|(\Iv_g-\Iv^{(n)}_g)b \| \Big) = 0,
\]
then $(\psi,\Iv)\wc \{(\psi^{(n)},\Iv^{(n)}) \}_{n\in\IN}$.
\item \label{prop:wc-observations:2}
If there is a sequence of unitaries $u_n\in\CU(\CM(B))$ such that $\dst \psi=\lim_{n\to\infty} \ad(u_n)\circ\phi$ in the point-strict topology and $\dst \Iv_\bullet=\lim_{n\to\infty} u_n\Iu_\bullet\beta_\bullet(u_n)^*$ strictly and uniformly over compact sets, then $(\psi,\Iv)\wc(\phi,\Iu)$.
\end{enumerate}
If furthermore $\beta$ is strongly stable, then
\begin{enumerate}[label=\textup{(\roman*)},leftmargin=*,resume]
\item \label{prop:wc-observations:3}
One has $(\psi,\Iv)\wc\{(\psi^{(n)},\Iv^{(n)}) \}_{n\in\IN}$ if and only if $\dst(\psi,\Iv)\wc \bigoplus_{n=1}^\infty (\psi^{(n)},\Iv^{(n)})$.
\item \label{prop:wc-observations:4}
One has $(\phi^\infty,\Iu^\infty) \wc (\phi,\Iu)$.
\end{enumerate}
\end{prop}
\begin{proof}
\ref{prop:wc-observations:1}:
Let $K\subseteq G$, $\CF\subset A$, $b\in B$ and $\eps>0$ be given.
By assumption, using that $\{ \beta_g(b) : g\in K\}$ is compact, we may choose $j\geq 1$ such that
\[
\max_{g\in K} \| b^*\Iv_g\beta_g(b)-b^*\Iv^{(j)}_g\beta_g(b) \Big\| \leq \eps
\]
and
\[
\max_{a\in\CF} \| b^*\psi(a)b - b^*\psi^{(j)}(a)b \Big\| \leq \eps.
\]
Thus $\ell=N=1$ and $c_{1,1}=b$ have the desired property as required by \autoref{def:wc}, which shows $(\psi,\Iv)\wc \{(\psi^{(n)},\Iv^{(n)}) \}_{n\in\IN}$.

\ref{prop:wc-observations:2}:
It is trivial that two unitarily equivalent cocycle representations are weakly contained in each other.
Thus the claim is a straightforward consequence of \ref{prop:wc-observations:1}.

\ref{prop:wc-observations:3}:
Let $t_k\in\CM(B)^\beta$ be a sequence of isometries such that $\sum_{k=1}^\infty t_kt_k^*=\eins$ in the strict topology.
Denote $(\Psi,\IV)=\bigoplus_{k=1}^\infty (\psi^{(k)},\Iv^{(k)})$, where the sum is formed via this sequence.
Suppose a quadruple $(K,\CF,b,\eps)$ is given.
Without loss of generality, let us assume that $\CF$ is a subset of the unit ball of $A$.

For the ``if'' part, suppose we have contractions $d_1,\dots,d_n\in B$ with
\[
\max_{g\in K} \Big\| b^*\Iv_g\beta_g(b)-\sum_{j=1}^n d_j^*\IV_g\beta_g(d_j) \Big\| \leq \eps
\]
and
\[
\max_{a\in\CF} \Big\| b^*\psi(a)b - \sum_{j=1}^n d_j^*\Psi(a)d_j \Big\| \leq \eps.
\]
Choose $N\in\IN$ such that
\begin{equation} \label{eq:wc-obs:1}
\max_{1\leq j\leq n} \|(\eins-\sum_{k=1}^N t_kt_k^*)d_j\| \leq \eps/2n.
\end{equation}
Then if we set $c_{j,k}=t_k^*d_j$ for $j=1,\dots,n$ and $k=1,\dots,N$, we observe for all $g\in K$ that
\[
\begin{array}{ccl}
\dst \sum_{j=1}^n\sum_{k=1}^N c_{j,k}^*\Iv_g^{(k)} \beta_g(c_{j,k}) 
&=& \dst\sum_{j=1}^n\sum_{k=1}^N d_j^* t_k \Iv^{(k)}_g t_k^* \beta_g(d_j) \\
&\stackrel{\eqref{eq:wc-obs:1}}{=}_{\makebox[0pt]{\footnotesize\hspace{-3mm}$\eps$}}& \dst \sum_{j=1}^n d_j^*\IV_g\beta_g(d_j) \ =_\eps \ b^*\Iv_g\beta_g(b).
\end{array}
\]
Similarly we compute for all $a\in \CF$ that
\[
\begin{array}{ccl}
\dst \sum_{j=1}^n\sum_{k=1}^N c_{j,k}^*\psi^{(k)}(a)c_{j,k} 
&=& \dst\sum_{j=1}^n\sum_{k=1}^N d_j^* t_k \psi^{(k)}(a) t_k^* d_j \\
&\stackrel{\eqref{eq:wc-obs:1}}{=}_{\makebox[0pt]{\footnotesize\hspace{-3mm}$\eps$}}& \dst \sum_{j=1}^n d_j^*\Psi(a) d_j \ =_{\makebox[0pt]{\footnotesize\hspace{2mm}$\eps$}} \ \ b^*\psi(a)b.
\end{array}
\]
This verifies $(\psi,\Iv)\wc \{(\psi^{(n)},\Iv^{(n)}) \}_{n\in\IN}$.

For the ``only if'' part, suppose we have found $n,N\geq 1$ and $c_{j,k}\in B$ for the quadruple $(K,\CF,b,\eps)$ as required by \autoref{def:wc}.
Set $d_{j,k}=t_kc_{j,k}\in B$.
Similarly as above we have for all $a\in A$
\[
\sum_{j=1}^n\sum_{k=1}^N d_{j,k}^*\Psi(a)d_{j,k} = \sum_{j=1}^n \sum_{k=1}^N c_{j,k}^*\psi^{(k)}(a)c_{j,k} =_\eps b^*\psi(a)b
\]
and for all $g\in K$ that
\[
\sum_{j=1}^n\sum_{k=1}^N d_{j,k}^*\IV_g\beta_g(d_{j,k}) = \sum_{j=1}^n\sum_{k=1}^N c_{j,k}^*\Iv^{(k)}_g\beta_g(c_{j,k}) =_\eps b^*\Iv_g\beta_g(b).
\]
This shows $(\psi,\Iv)\wc(\Psi,\IV)$.

\ref{prop:wc-observations:4} is a trivial consequence of \ref{prop:wc-observations:3}.
\end{proof}


We need the following fundamental technical lemma, which we will use in a crucial step later.
It is a generalization of a statement that appeared in the literature before, usually in a somewhat different form, both in the context of $KK$-theory and the classification theory of purely infinite \cstar-algebras; see for example \cite[Lemma 6.3.7]{Rordam} or \cite[Lemma 2.4]{KirchbergRordam05}.

\begin{lemma} \label{lem:the-O2-sum-lemma}
Let $B$ be any \cstar-algebra and $\kappa\in\Aut(B)$ an automorphism, which we extend canonically to an automorphism on $\CM(B)$.
Let $b_1,b_2\in \CM(B)$ be two elements, and $s\in\CM(B)$ an isometry.
Suppose that $r_1,r_2\in\CM(B)^\kappa$ are two isometries fixed by $\kappa$ that satisfy $r_1r_1^*+r_2r_2^*=\eins$.
Furthermore, suppose that $[b_1,r_j]=0=[b_1,r_j^*]$ for $j=1,2$.
Then the element
\[
u = (r_1r_1^* + r_2sr_2^*)s^*+r_2(\eins-ss^*) \in \CM(B)
\]
is a unitary, and it satisfies
\[
\| ub_2\kappa(u)^* - (b_1\oplus_{r_1,r_2} b_2) \| \leq 5\cdot\max\Big\{ \| s b_1 - b_2 \kappa(s) \| , \| \kappa(s) b_1^* - b_2^* s \| \Big\}.
\]
Furthermore, if $sb_1-b_2\kappa(s)\in B$ and $\kappa(s)b_1^*-b_2^*s\in B$, then also 
\[
 u b_2 \kappa(u)^*-(b_1\oplus_{r_1,r_2} b_2) \in B.
\]
\end{lemma}
\begin{proof}
First, let us verify that $u$ is indeed a unitary.
We have
\[
u^*u = s(r_1r_1^*+r_2s^*sr_2^*)s^* + (\eins-ss^*) = \eins
\]
and
\[
\begin{array}{ccl}
uu^* &=& r_1r_1^*s^*sr_1r_1^* + r_2sr_2^*s^*sr_2s^*r_2^*+r_2(\eins-ss^*)r_2^* \\
&=& r_1r_1^* + r_2r_2^* \ = \ \eins.
\end{array}
\]
Set $\eps = \max\big\{ \| s b_1 - b_2 \kappa(s) \| , \| \kappa(s) b_1^* - b_2^* s \| \big\}$.
We compute
\[
\begin{array}{ll}
 \multicolumn{2}{l}{ \big( b_1\oplus_{r_1,r_2} b_2 \big) \kappa(u) } \\
=& \big( r_1 b_1 r_1^* + r_2 b_2 r_2^*\big) \kappa(u) \\
=& r_1 b_1 r_1^* \kappa(s)^* + r_2 b_2 r_2^* \kappa(u) \\
=& r_1 r_1^* b_1 \kappa(s)^* + r_2 b_2(\eins-\kappa(s)\kappa(s)^*) + r_2 b_2 \kappa(s) r_2^* \kappa(s)^* \\
=_{2\eps}& r_1 r_1^* b_1 \kappa(s)^* + r_2(b_2-s b_1 \kappa(s)^*) + r_2 s b_1 r_2^* \kappa(s)^* \\
=&  r_1 r_1^* b_1 \kappa(s)^* + r_2(b_2-s b_1 \kappa(s)^*) + r_2 s r_2^* b_1 \kappa(s)^* \\
=_{2\eps}& r_1 r_1^* b_1 \kappa(s)^* + r_2(\eins-ss^*)b_2 + r_2 s r_2^* s^* b_2 \\
=_{\eps}& (r_1 r_1^* s^* + r_2(\eins-ss^*) + r_2 s r_2^* s^* ) b_2 \ = \ u b_2.
\end{array}
\]
We also see from this calculation that if $sb_1-b_2\kappa(s)\in B$ and $\kappa(s)b_1^*-b_2^*s\in B$, then the intermediate differences in these calculations are also elements of $B$.
We may now multiply everything with $\kappa(u)^*$ from the right and obtain the desired statement.
\end{proof}

This leads to the following sufficient criterion for the absorption of cocycle representations:

\begin{lemma} \label{lem:sufficient-criterion-absorption}
Let $(\phi,\Iu), (\psi,\Iv): (A,\alpha)\to (\CM(B),\beta)$ be two cocycle representations.
Suppose there is a sequence of isometries $s_n\in\CM(B)$ satisfying
  \begin{itemize}
  \item $\dst \lim_{n\to\infty} \| s_n\psi(a)-\phi(a)s_n \|=0$ for all $a\in A$;
  \item $s_n\psi(a)-\phi(a)s_n\in B$ for all $a\in A$ and $n\geq 1$;
  \item $\dst \lim_{n\to\infty} \sup_{g\in K} \| s_n\Iv_g - \Iu_g\beta_g(s_n) \| = 0$ for all compact sets $K\subseteq G$;
  \item The map $[g\mapsto s_n\Iv_g - \Iu_g\beta_g(s_n)]$ takes values in $B$ for all $n\geq 1$.
  \item $s_{n+1}^*s_n=0$ for all $n\geq 1$.
  \end{itemize}
Suppose that there exists a unital inclusion $\CO_2\subseteq\CM(B)^\beta$ that commutes with the ranges of $\psi$ and $\Iv$.
Then $(\phi,\Iu)\asymp (\psi,\Iv)\oplus(\phi,\Iu)$.
\end{lemma}
\begin{proof}
We may employ the same trick as in \cite[Lemmas 2.2+2.3]{DadarlatEilers01} and (as a consequence of the last bullet point) extend the sequence $(s_n)_n$ to a norm-continuous path of isometries $(s_t)_{t\geq 1}$ by defining 
\[
s_{n+t} = (1-t)^{1/2} s_n + t^{1/2}s_{n+1} \quad\text{for all } n\geq 1 \text{ and } t\in [0,1].
\]
By the properties of $s_n$ assumed above, this continuous path is seen to satisfy the following properties:
  \begin{equation} \label{eq:sca:1}
  \lim_{t\to\infty} \| s_t\psi(a)-\phi(a)s_t \| =0 \quad\text{for all } a\in A;
  \end{equation}
  \begin{equation} \label{eq:sca:2}
  s_t\psi(a)-\phi(a)s_t\in B \quad\text{for all } a\in A \text{ and } t\geq 1;
  \end{equation}
  \begin{equation} \label{eq:sca:3}
  \lim_{t\to\infty} \sup_{g\in K} \| s_t\Iv_g - \Iu_g\beta_g(s_t) \| = 0 \quad\text{for all compact sets } K\subseteq G;
  \end{equation}
  \begin{equation} \label{eq:sca:4}
  [(t,g) \mapsto s_t\Iv_g-\Iu_g\beta_g(s_t)] \text{ takes values in } B.
  \end{equation}
By assumption, we may pick a pair of isometries $r_1,r_2\in\CM(B)^\beta$ with $r_1r_1^*+r_2r_2^*=\eins$ commuting with the ranges of $\psi$ and $\Iv$.
Define a norm-continuous path of unitaries $u: [1,\infty)\to\CU(\CM(B))$ via
\[
u_t = (r_1 r_1^* + r_2 s_t r_2^*)s_t^*+r_2(\eins-s_ts_t^*) \in \CM(B).
\]
By \autoref{lem:the-O2-sum-lemma}, these are indeed unitaries.
When we apply the lemma for $\id_B$ in place of $\kappa$, then conditions \eqref{eq:sca:1} and \eqref{eq:sca:2} ensure that we have
\[
\lim_{t\to\infty} \|u_t\phi(a)u_t^*-(\psi\oplus_{r_1,r_2}\phi)(a)\|=0
\]
for all $a\in A$, and that the expressions appearing in the norm are elements in $B$.
When we apply the lemma for $\kappa\in\{\beta_g\}_{g\in G}$, we can use conditions \eqref{eq:sca:3} and \eqref{eq:sca:4} along with the identities $\Iv_g^*=\beta_g(\Iv_{g^{-1}})$ and $\Iu_g^*=\beta_g(\Iu_{g^{-1}})$ to compute
\[
\begin{array}{cl}
\multicolumn{2}{l}{\dst \lim_{t\to\infty}\max_{g\in K} \|u_t\Iu_g\beta_g(u_t)^*-(\Iv_g\oplus_{r_1,r_2}\Iu_g)\| } \\
\leq & \dst 5 \lim_{t\to\infty} \sup_{g\in K} \Big( \|s_t\Iv_g - \Iu_g\beta_g(s_t)\| + \|\beta_g(s_t)\Iv_g^* - \Iu_g^* s_t\| \Big) \\
=& \dst 5 \lim_{t\to\infty} \sup_{g\in K} \Big( \|s_t\Iv_g - \Iu_g\beta_g(s_t)\| + \|s_t\Iv_{g^{-1}} - \Iu_{g^{-1}} \beta_{g^{-1}} s_t\| \Big) \\
=& 0
\end{array}
\]
for every compact set $K\subseteq G$.
Furthermore the expressions appearing in the norms are elements of $B$.
This finishes the proof.
\end{proof}

\begin{rem}
Since it is useful for our companion paper \cite{GabeSzabo24kp}, let us observe a few things about the statement in \autoref{lem:sufficient-criterion-absorption}.
Firstly, if there exists a unital inclusion $\CO_\infty\subseteq\CM(B)^\beta$ that commutes with the ranges of $\phi$ and $\Iu$, then modulo passing to a different sequence of isometries, the last bullet point is redundant.
This is because given isometries $R_1, R_2$ in this copy of $\CO_\infty$ with $R_1^*R_2=0$, one may replace $s_n$ by $R_1s_n$ when $n$ is odd and by $R_2s_n$ when $n$ is even.
This does not disturb any of the first four bullet points, but forces the property $s_{n+1}^*s_n=0$ in the last bullet point.

Secondly, we may consider what happens in a special case of \autoref{lem:sufficient-criterion-absorption}.
In addition to the inclusion $\CO_\infty\subseteq\CM(B)^\beta$ above, assume that $(\phi,\Iu)$ and $(\psi,\Iv)$ are proper cocycle morphisms (see \cite[Definition 1.10(iii)]{Szabo21cc}), meaning that $\phi$ and $\psi$ have values in $B$, and $\Iu$ and $\Iv$ have values in $\CU(\eins+B)$.
In this case, we may assume $s_1=R_1$ and carry out the rest of the proof as before.
This gives rise to a norm-continuous path of unitaries $u: [1,\infty)\to\CU(\CM(B))$ witnessing $(\phi,\Iu)\oplus(\psi,\Iv)\asymp (\phi,\Iu)$ that additionally satisfies $u_1\in\CM(B)^\beta$.
Since the cocycles $\Iu$ and $\Iv$ have values in $\CU(\eins+B)$, the relation $u_t\Iu_g\beta_g(u_t)^*-(\Iv_g\oplus\Iu_g)\in B$ for all $g\in G$ amounts to $u_t\in\CM^\beta(B)$.
\end{rem}

\begin{lemma} \label{lem:abs-inf-repeat}
Suppose $\beta$ is strongly stable.
Let $(\phi,\Iu), (\psi,\Iv): (A,\alpha)\to (\CM(B),\beta)$ be two cocycle representations.
The following are equivalent:
\begin{enumerate}[label=\textup{(\roman*)},leftmargin=*]
\item \label{lem:abs-inf-repeat:1}
$(\psi,\Iv)$ is contained in $(\phi,\Iu)$ at infinity;
\item \label{lem:abs-inf-repeat:2}
$(\psi^\infty,\Iv^\infty)$ is contained in $(\phi,\Iu)$ at infinity;
\item \label{lem:abs-inf-repeat:3}
there exists an isometry $S\in\CM(B)$ such that
\[
\psi^\infty(a)-S^*\phi(a)S\in B \quad\text{for all } a\in A
\]
and the map
\[
\big[ g\mapsto \Iv_g^\infty-S^*\Iu_g\beta_g(S) \big]
\]
takes values in $B$ and is norm-continuous.
\item \label{lem:abs-inf-repeat:4}
$(\phi,\Iu)\asymp (\psi^\infty,\Iv^\infty)\oplus(\phi,\Iu)$.
\end{enumerate}
\end{lemma}
\begin{proof}
We will proceed in the order \ref{lem:abs-inf-repeat:1}$\Rightarrow$\ref{lem:abs-inf-repeat:2}$\Rightarrow$\ref{lem:abs-inf-repeat:3}$\Rightarrow$\ref{lem:abs-inf-repeat:1} and \ref{lem:abs-inf-repeat:3}$\Leftrightarrow$\ref{lem:abs-inf-repeat:4}.
For the rest of the proof, we let $t_n\in\CM(B)^\beta$ be a sequence of isometries with $\sum_{n=1}^\infty t_nt_n^*=\eins$ in the strict topology.
All infinite repeats are understood to be formed via this sequence.

\ref{lem:abs-inf-repeat:1}$\Rightarrow$\ref{lem:abs-inf-repeat:2}:
Let compact sets $K\subseteq G$, $\CF^A\subset A$, and $\eps> 0$ be given.
Without loss of generality assume $K=K^{-1}$ and $\CF^A=\CF_A^*$ is in the unit ball of $A$.

Suppose that $\CF^B\subset B$ is a compact set and $b\in B$ is a contraction.

Let $N\in \mathbb N$ such that $\| b - \sum_{j=1}^N t_jt_j^\ast b \| \leq \eps/8$ and pick a positive contraction $b_0\in B$ such that $\| b_0 t_j^\ast b- t_j^\ast b\| \leq \eps/(8N)$ for $j=1,\dots,N$.
Then
\begin{equation} \label{eq:air:1i2-1}
\|b-b_{0,N}b\|\leq \eps/4 \quad\text{for}\quad b_{0,N} = \sum_{j=1}^N t_jb_0t_j^*.
\end{equation}
Using the statement in \ref{lem:abs-inf-repeat:1}, we may recursively choose contractions $x_1,\dots,x_N\in B$ satisfying
\begin{equation} \label{eq:air:1i2-2}
\max_{1\leq j\leq N} \max_{g\in K} \| b_0\Iv_g\beta_g(b_0)-x_j^*\Iu_g\beta_g(x_j) \| \leq \eps/4;
\end{equation}
\begin{equation} \label{eq:air:1i2-3}
\max_{1\leq j\leq N} \max_{a\in\CF^A} \| b_0\psi(a)b_0 - x_j^*\phi(a)x_j \| \leq \eps/4;
\end{equation}
\begin{equation} \label{eq:air:1i2-4}
\max_{1\leq j\leq N} \max_{1\leq l<j} \Big( \max_{a\in\CF^A} \| x_j^*\phi(a)x_l \| + \max_{g\in K} \| x_j^*\Iu_g\beta_g(x_l) \| \Big) \leq 2^{-N^2}\eps/4;
\end{equation}
\begin{equation} \label{eq:air:1i2-5}
\max_{1\leq j\leq N} \max_{c\in\CF^B} \| x_j^*c \| \leq \eps/N.
\end{equation}
We define an element $x\in B$ via $x=\sum_{j=1}^N x_jt_j^*$.
We observe $\dst\max_{c\in\CF^B} \|x^*c\| \leq \eps$ as a consequence of \eqref{eq:air:1i2-5}.
We have for all $g\in K$ that
\[
\begin{array}{cclcl}
x^*\Iu_g\beta_g(x) &=& \dst \sum_{j,k=1}^N t_j x_j^* \Iu_g \beta_g(x_k) t_k^* 
&\stackrel{\eqref{eq:air:1i2-4}}{=}_{\makebox[0pt]{\footnotesize $\eps/4$}}& \dst \sum_{j=1}^N t_j x_j^* \Iu_g \beta_g(x_j) t_j^* \\
&\stackrel{\eqref{eq:air:1i2-2}}{=}_{\makebox[0pt]{\footnotesize $\eps/4$}}& \dst \sum_{j=1}^N t_j b_0\Iv_g\beta_g(b_0) t_j^*
&=& b_{0,N}\Iv^\infty_g \beta_g(b_{0,N}).
\end{array}
\]
If we consider this computation, we shall multiply both sides with $b^*$ from the left and with $\beta_g(b)$ from the right, and use \eqref{eq:air:1i2-1} to obtain
\[
\max_{g\in K} \| b^*\Iv_g^\infty\beta_g(b)-(xb)^*\Iu_g\beta_g(xb) \| \leq \eps.
\]
Moreover we have for all $a\in\CF^A$ that
\[
\begin{array}{cclcl}
x^*\phi(a)x &=& \dst \sum_{j,k=1}^N t_j x_j^* \phi(a) x_k t_k^* 
&\stackrel{\eqref{eq:air:1i2-4}}{=}_{\makebox[0pt]{\footnotesize $\eps/4$}}& \dst \sum_{j=1}^N t_j x_j^* \phi(a) x_j t_j^* \\
&\stackrel{\eqref{eq:air:1i2-3}}{=}_{\makebox[0pt]{\footnotesize $\eps/4$}}& \dst \sum_{j=1}^N t_j b_0\psi(a)b_0 t_j^* 
&=& b_{0,N} \psi^\infty(a) b_{0,N}.
\end{array}
\]
Similarly as above we consider this computation, multiply both sides with $b^*$ from the left and with $b$ from the right, and use \eqref{eq:air:1i2-1} to obtain
\[
\max_{a\in\CF^A} \| b^*\psi^\infty(a)b - (xb)^*\phi(a)(xb) \| \leq \eps.
\]
Clearly we have $\max_{c\in\CF^B} \|(xb)^*c\| \leq \max_{c\in\CF^B} \|x^*c\| \leq \eps$.
Thus the element $xb\in B$ satisfies the properties as required by \autoref{def:contained-at-infty} in place of $x$.
This verifies condition \ref{lem:abs-inf-repeat:2}.

\ref{lem:abs-inf-repeat:2}$\Rightarrow$\ref{lem:abs-inf-repeat:3}:
Let $\CF^A\subset A$ be a self-adjoint compact set in the unit ball whose closed linear span is $A$.
This exists by the separability of $A$.
Fix a constant $0<\eps<\frac12$.
Choose an increasing sequence $1_G\in K_n\subseteq G$ of compact sets with $G=\bigcup_{n\geq 1} \operatorname{int}(K_n)$.

By \autoref{lem:Kasparov}, $B$ admits a countable, increasing, approximately $\beta^{\Iv^\infty}$-invariant approximate unit $(e_n)_{n\geq 1}$ of positive contractions that is quasicentral relative to $\psi^\infty(A)$.
We denote $f_1=e_1^{1/2}$ and $f_n=(e_{n}-e_{n-1})^{1/2}$ for $n\geq 2$.
By passing to a subsequence of the $e_n$, if necessary, we may assume 
\begin{equation} \label{eq:air:2i3-1}
\max_{g\in K_n} \| f_n-\beta^{\Iv^\infty}_g(f_n) \| \leq 2^{-(n+1)}\eps
\end{equation}
and
\begin{equation} \label{eq:air:2i3-2}
\max_{a\in \CF^A} \| [\psi^\infty(a),f_n] \| \leq 2^{-n}\eps
\end{equation}
for all $n\geq 1$.
By our choice of $f_n$, we can apply condition \ref{lem:abs-inf-repeat:2} inductively and find a sequence of contractions $x_n\in B$ satisfying the following properties for all $n\geq 1$:
\begin{equation} \label{eq:air:2i3-3}
\max_{g\in K_n} \| f_n\Iv^\infty_g\beta_g(f_n) - x_n^*\Iu_g\beta_g(x_n) \| \leq 2^{-(n+1)}\eps;
\end{equation}
\begin{equation} \label{eq:air:2i3-4}
\max_{g\in K_n} \Big( \| x_n^*\Iu_g\beta_g(x_j) \| + \| \beta_g(x_n)^* \Iu^*_g x_j \| \Big) \leq 2^{-(2+j+n)}\eps,\quad j<n;
\end{equation}
\begin{equation} \label{eq:air:2i3-5}
\max_{a\in \CF^A} \| f_n\psi^\infty(a)f_n - x_n^*\phi(a)x_n \|\leq 2^{-n}\eps;
\end{equation}
\begin{equation} \label{eq:air:2i3-6}
\max_{a\in \CF^A} \| x_n^*\phi(a)x_j \| \leq 2^{-(1+n+j)}\eps,\quad j<n;
\end{equation}
\begin{equation} \label{eq:air:2i3-7}
\| x_n^* e_j \| \leq 2^{-n}\eps,\quad j\leq n.
\end{equation}
We claim that the series $X = \sum_{n=1}^\infty x_n$ converges in the strict topology and therefore yields an element $X\in\CM(B)$.
Indeed, by using \eqref{eq:air:2i3-3} (with $g=1$), we have for any pair of natural numbers $k\leq \ell$ that
\[
\Big\| \sum_{n=k}^\ell x_n^*x_n - \sum_{n=k}^\ell f_n^2 \Big\| \leq \sum_{n=k}^\ell 2^{-(n+1)}\eps \leq 2^{-k}\eps.
\]
Since we have $\sum_{j=1}^\infty f_j^2=\eins$ in the strict topology, we conclude that the sequence of partial sums $\ell\mapsto \sum_{n=1}^\ell x_n^*x_n$ satisfies the Cauchy criterion in the strict topology.
If $b\in B$ is a contraction, this implies for all pairs of natural numbers $k\leq\ell$ that
\[
\begin{array}{ccl}
\dst \Big\| \sum_{n=k}^\ell x_n b \Big\|^2 &=& \dst\Big\| b^*\Big( \sum_{j,n=k}^\ell x_n^* x_j \Big)b \Big\| \\
&\stackrel{\eqref{eq:air:2i3-4}}{\leq}& \dst \sum_{j,n=k\atop j<n}^\ell 2^{-(2+j+n)}\eps+ \Big\| b^*\Big( \sum_{n=k}^\ell x_n^* x_n \Big)b \Big\| \\
&\leq& \dst 2^{-k}\eps+ \Big\| b^*\Big( \sum_{n=k}^\ell x_n^* x_n \Big)b \Big\|.
\end{array}
\]
In particular, the sequence $\ell\mapsto \big(\sum_{n=1}^\ell x_n\big)b$ satisfies the Cauchy criterion.
Note that this last computation is valid even when choosing $b=\eins$, so we see that $\ell\mapsto\big(\sum_{n=1}^\ell x_n\big)$ is a norm-bounded sequence.
On the other hand, condition \eqref{eq:air:2i3-7} implies that the sequence $\ell\mapsto \big(\sum_{n=1}^\ell x_n^*\big) e_j$ satisfies the Cauchy criterion for all $j\geq 1$.
Since the sum is uniformly bounded in norm and the $e_j$ were an approximate unit of $B$ by assumption, we may conclude with a standard $\eps/2$-argument that the sequence $\ell\mapsto \big(\sum_{n=1}^\ell x_n^*\big) b$ satisfies the Cauchy criterion for all $b\in B$.
We may hence conclude that the series in question converges in the strict topology, or in other words, that $X$ exists.

Next we check that $X$ satisfies similar properties as required by the elements $S$ in the statement of \ref{lem:abs-inf-repeat:3}.
As in the proof of \cite[Theorem 5]{Kasparov80}, we check for all $a\in\CF^A$ the equality modulo $B$
\[
\begin{array}{cclcl}
X^*\phi(a)X &=& \dst \sum_{k,n=1}^\infty x_k^*\phi(a)x_n
&\stackrel{\eqref{eq:air:2i3-6}}{\equiv}& \dst \sum_{n=1}^\infty x_n^*\phi(a)x_n  \\
&\stackrel{\eqref{eq:air:2i3-5}}{\equiv}& \dst \sum_{n=1}^\infty f_n\psi^\infty(a)f_n  & \stackrel{\eqref{eq:air:2i3-2}}{\equiv} & \psi^\infty(a).
\end{array}
\]
Here we used that the pointwise differences between the involved terms in each step are absolutely convergent series in $B$.
We see that $X^*\phi(a)X-\psi^\infty(a)\in B$ holds for all $a\in\CF^A$.
By the choice of $\CF^A$, the latter relation extends to all $a\in A$.
Next, given any element $g\in G$, let $j\geq 1$ be the smallest number with $g\in K_j$.
Then we have
\[
\begin{array}{ccl}
X^*\Iu_g \beta_g(X) &\hspace{4mm}=\hspace{4mm}& \dst \sum_{k,n=1}^\infty x_k^*\Iu_g \beta_g(x_n) \\
&\stackrel{\eqref{eq:air:2i3-4}}{=}_{\makebox[0pt]{\footnotesize $2^{-j}\eps$}} & \dst \sum_{k,n<j} x_k^*\Iu_g \beta_g(x_n) + \sum_{n\geq j} x_n^*\Iu_g\beta_g(x_n) \\
&\stackrel{\eqref{eq:air:2i3-3}}{=}_{\makebox[0pt]{\footnotesize $2^{-j}\eps$}} & \dst \sum_{k,n<j} x_k^*\Iu_g \beta_g(x_n) + \sum_{n\geq j} f_n\Iv^\infty_g\beta_g(f_n) \\
&\stackrel{\eqref{eq:air:2i3-1}}{=}_{\makebox[0pt]{\footnotesize $2^{-j}\eps$}} & \dst \sum_{k,n<j} x_k^*\Iu_g \beta_g(x_n) + \sum_{n\geq j} f_n^2\Iv^\infty_g \\
&=& \dst \Iv^\infty_g + \Big( \sum_{k,n<j} x_k^*\Iu_g \beta_g(x_n) - \sum_{n<j} f_n^2\Iv^\infty_g \Big).
\end{array}
\]
In each of these approximation steps, we see that all the maps on $G$ defined by the intermediate differences take values in $B$ and are continuous.
So indeed the assignment $[g\mapsto \Iv_g^\infty-X^*\Iu_g \beta_g(X)]$ takes values in $B$ and is continuous.
If furthermore $g=1_G$, then $j=1$ above, so the very last term in the above calculation vanishes and we conclude $X^*X\in \eins+B$ and
\[
\|\eins-X^*X\|< 2\eps<1.
\]
We claim that the isometry $S=X|X|^{-1}$ does the trick.
From the above we can immediately conclude $\psi^\infty(a)-S^*\phi(a)S\in B$ for all $a\in A$, as well as $\Iv^\infty_g-S^*\Iu_g\beta_g(S)\in B$ for all $g\in G$.
Since
\[
\begin{array}{cl}
\multicolumn{2}{l}{ \Iv^\infty_g-S^*\Iu_g\beta_g(S) }\\
=& \Iv^\infty_g-|X|^{-1}X^*\Iu_g\beta_g(X|X|^{-1})  \\
=& \Iv^\infty_g - X^*\Iu_g\beta_g(X) + X^*\Iu_g\beta_g(X(\eins-|X|^{-1}))  + (\eins-|X|^{-1})X^*\Iu_g\beta_g(X|X|^{-1}) 
\end{array}
\]
and $|X|\in\eins+B$, the strict continuity of the cocycles forces the assignment $[ g\mapsto \Iv^\infty_g-S^*\Iu_g\beta_g(S)]$ to be continuous in norm.

\ref{lem:abs-inf-repeat:3}$\Rightarrow$\ref{lem:abs-inf-repeat:1}:
Let $S\in\CM(B)$ be an isometry as in the statement of \ref{lem:abs-inf-repeat:3}.
Let $b\in B$ be a contraction. 
Recall that $\psi^\infty(\bullet) = \sum_{j=1}^\infty t_j \psi(\bullet) t_j^*$ and $\Iv^\infty_\bullet = \sum_{j=1}^\infty t_j \Iv_\bullet t_j^*$.
Define $x_n = S t_{n} b$.
Then clearly $x_n^*b\to 0$.

As $t_{n}^* \psi^\infty(a) t_{n} = \psi(a)$ for all $n\geq 1$ and $a\in A$, we get
\[
\| x_n^\ast \phi(a) x_n - b^\ast \psi(a) b \|  = \| b^\ast t_{n}^* (\underbrace{ S^* \phi(a) S - \psi^\infty(a)}_{\in B}) t_{n} b \| \to 0.
\]
Similarly we get for any compact set $K\subseteq G$ that
\[
\max_{g\in K} \| x_n^\ast \Iu_g \beta_g(x_n) - b^* \Iv_g \beta_g(b)\| = \max_{g\in K} \| b^* t_{n}^\ast ( S^\ast \Iu_g \beta_g(S) - \Iv_g^\infty) t_{n} b \| \to 0. 
\]
Here we used the continuity assumption in \ref{lem:abs-inf-repeat:3} to conclude that the terms appearing in the middle bracket form a compact subset inside $B$.
We see that $(\psi , \Iv)$ is contained in $(\phi, \Iu)$ at infinity.

\ref{lem:abs-inf-repeat:3}$\Rightarrow$\ref{lem:abs-inf-repeat:4}:
We consider the double-infinite repeat 
\[
(\psi^\infty)^\infty(a)=\sum_{n,m=1}^\infty t_nt_m\psi(a)t_m^*t_n^*
\] 
for all $a\in A$, and analogously define $(\Iv^\infty)^\infty$.
Since $(\psi^\infty,\Iu^\infty)$ is unitarily equivalent to $\big( (\psi^\infty)^\infty, (\Iu^\infty)^\infty \big)$, we may conclude the existence of an isometry $S\in\CU(\CM(B))$ such that 
\[
S^*\phi(a)S-(\psi^\infty)^\infty(a)\in B \quad\text{and}\quad S^*\Iu_g\beta_g(S)-(\Iv^\infty)_g^\infty\in B
\]
for all $a\in A$ and $g\in G$, and such that the right-most term defines a norm-continuous map on $G$.
For any $n\geq 1$, we have $t_n^*(\psi^\infty)^\infty(a)t_n=\psi^\infty(a)$ and $t_n^*(\Iv^\infty)_g^\infty t_n=\Iv^\infty_g$.
Furthermore we have $t_n^*bt_n\to 0$ for all $b\in B$.
This allows us to conclude that the sequence of isometries $s_n=St_n$ satisfies
\[
B \ \ni \ t_n^*\big( S^*\phi(a)S-(\psi^\infty)^\infty(a) \big) t_n = s_n^*\phi(a)s_n-\psi^\infty(a) \stackrel{n\to\infty}{\longrightarrow} 0
\]
for all $a\in A$.
Analogously we conclude
\[
B \ \ni \ t_n^*\big( S^*\Iu_g\beta_g(S)-(\Iv^\infty)_g^\infty \big) t_n = s_n^*\Iu_g \beta_g(s_n)-\Iv^\infty_g \stackrel{n\to\infty}{\longrightarrow} 0
\]
for all $g\in G$, and this convergence is uniform over compact subsets of $G$.

As in \cite{Kasparov80} or the proof of \cite[Proposition 2.15]{DadarlatEilers02}, we observe
\[
\begin{array}{cl}
\multicolumn{2}{l}{ (s_n\psi^\infty(a)-\phi(a)s_n)^*(s_n\psi^\infty(a)-\phi(a)s_n) } \\
=& (s_n^*\phi(a^*a)s_n-\psi^\infty(a^*a)) + (\psi^\infty(a^*)-s_n^*\phi(a^*)s_n)\psi^\infty(a) + \\
& + \psi^\infty(a^*)(\psi^\infty(a)-s_n^*\phi(a)s_n),
\end{array}
\]
which leads to $\dst\lim_{n\to\infty} \| s_n\psi^\infty(a)-\phi(a)s_n \|=0$ for all $a\in A$.
We also see that $s_n\psi^\infty(a)-\phi(a)s_n \in B$ for all $a\in A$.
Similarly we observe 
\[
\begin{array}{cl}
\multicolumn{2}{l}{ (s_n\Iv^\infty_g-\Iu_g\beta_g(s_n))^*(s_n\Iv^\infty_g-\Iu_g\beta_g(s_n)) } \\
=& 2\cdot\eins - \Iv^{\infty*}_g\cdot s_n^*\Iu_g\beta_g(s_n) - \beta_g(s_n)^*\Iu_g^* s_n \cdot\Iv^\infty_g,
\end{array}
\]
which leads to $\dst\lim_{n\to\infty} \max_{g\in K} \| s_n\Iv^\infty_g-\Iu_g\beta_g(s_n) \| = 0$ for all compact sets $K\subseteq G$.
We also see that the map $[g\mapsto s_n\Iv_g-\Iu_g\beta_g(s_n)]$ takes values in $B$.

By construction, the isometries $s_n$ have pairwise orthogonal ranges.
In order to complete the proof of the claim with \autoref{lem:sufficient-criterion-absorption}, it suffices to see that there is an inclusion $\CO_2\subseteq\CM(B)^\beta$ that commutes with the ranges of $\psi^\infty$ and $\Iv^\infty$.
This is indeed witnessed by the two isometries
\[
r_i = \sum_{n=1}^\infty t_{2n-i}t_{n}^* \in \CM(B)^\beta,\quad i=0,1,
\]
which satisfy the $\CO_2$-relation and commute with all such infinite repeats due to the simple observation that 
\[
\Big( \sum_{j=1}^\infty t_j x t_j^\ast \Big) r_i = \sum_{n=1}^\infty t_{2n-i} x t_n^\ast = r_i \Big( \sum_{j=1}^\infty t_j xt_j^\ast \Big),\quad i=0,1,
\]
for all $x\in \CM(B)$.

\ref{lem:abs-inf-repeat:4}$\Rightarrow$\ref{lem:abs-inf-repeat:3}:
Consider the isometry $t_\infty=\sum_{k=1}^\infty t_{k+1}t_k^*$, which fits into the equation $t_1t_1^*+t_\infty t_\infty^*=\eins$.
Let $U: [0,\infty)\to\CM(B)$ be a norm-continuous path witnessing the relation $(\phi,\Iu)\asymp (\phi,\Iu)\oplus (\psi^\infty,\Iv^\infty)$ in the sense of \autoref{def:various-equivalences}.
In particular, the unitary $U_0$ satisfies the properties
\[
U_0^*\phi(a)U_0 - \big( \phi(a) \oplus_{t_1,t_\infty} \psi^\infty(a)\big) \in B,\quad a\in A,
\]
and
\[
U_0^*\Iu_g\beta_g(U_0) -\big( \Iu_g \oplus_{t_1,t_\infty} \Iv_g^\infty) \in B,\quad g\in G.
\]
We set $S=U_0 t_\infty $.
This is an isometry, and it has the properties that
\[
S^*\phi(a)S - \psi^\infty(a) = t_\infty^*\big( U_0^*\phi(a)U_0 - \big( \phi(a) \oplus_{t_1,t_\infty} \psi^\infty(a)\big) \big) t_\infty \in B
\]
for all $a\in A$, and
\[
S^*\Iu_g\beta(S) - \Iv_g^\infty = t_\infty^* \big( U_0^*\Iu_g\beta_g(U_0) -\big( \Iu_g \oplus_{t_1,t_\infty} \Iv_g^\infty \big) \big) t_\infty \in B
\]
for all $g\in G$.
Moreover, because of \cite[Proposition 6.9]{Szabo21cc}, the middle expression in the brackets defines a norm-continuous map on $G$, hence the assignment $[g\mapsto S^*\Iu_g\beta(S) - \Iv_g^\infty]$ is also norm-continuous.
This finishes the proof.
\end{proof}

\begin{cor} \label{cor:coc-rep-absorption}
Suppose $\beta$ is strongly stable.
Let $(\phi,\Iu), (\psi,\Iv): (A,\alpha)\to (\CM(B),\beta)$ be two cocycle representations.
Then $(\psi,\Iv) \wc (\phi,\Iu)$ if and only if $(\phi^\infty,\Iu^\infty)\asymp (\psi^\infty,\Iv^\infty)\oplus(\phi^\infty,\Iu^\infty)$.
\end{cor}
\begin{proof}
For the rest of the proof, we fix a sequence of isometries $t_n\in\CM(B)^\beta$ with $\eins=\sum_{n=1}^\infty t_nt_n^*$ in the strict topology, and every instance of an infinite repeat shall use this sequence to define it.

Let us first show the ``if'' part.
Using \autoref{prop:wc-observations} and that $\asymp$ clearly implies weak containment, we can see that
\[
(\psi,\Iv) \wc (\psi^\infty,\Iv^\infty) \wc (\psi^\infty\oplus\phi^\infty,\Iv^\infty\oplus\Iu^\infty) \wc (\phi^\infty,\Iu^\infty) \wc (\phi,\Iu).
\]
So let us show the ``only if'' part.
Suppose $(\psi,\Iv) \wc (\phi,\Iu)$.
By \autoref{lem:abs-inf-repeat}, we need to show that $(\phi^\infty,\Iu^\infty)$ contains $(\psi,\Iv)$ at infinity.

Let $K\subseteq G, \CF\subset A$ and $\eps>0$ be given.
Without loss of generality, we assume $1_G\in K$.
Given a contraction $b\in B$, we may then choose contractions $d_1,\dots,d_n\in B$ as in \autoref{def:wc}.
Let $N\in\IN$ be large enough such that $\|(\eins-\sum_{n=1}^N t_nt_n^*)b\|\leq\eps$.
We define
\[
 x=\sum_{j=1}^n t_{N+j} d_j \in B.
\]
As we have $x^*(\eins-\sum_{n=1}^N t_nt_n^*)=x^*$ by construction, it satisfies $\| x^*b \| \leq \eps\|x\|$.
Given $g\in K$, we have
\[
x^*\Iu^\infty_g\beta_g(x)=\sum_{j=1}^n d_j^*\Iu_g\beta_g(d_j) =_\eps b^* \Iv_g \beta_g(b).
\]
The case $g=1_G$ shows $\|x^*x-b^*b\|\leq\eps$, so $x$ is close to a contraction when $\eps$ is chosen small.
For all $a\in\CF$, we have
\[
x^*\phi^\infty(a)x = \sum_{j=1}^n d_j^*\phi(a)d_j =_\eps b^*\psi(a)b.
\]
As the quadruple $(K,\CF,\eps,b)$ was arbitrary, we see that the condition as required by \autoref{def:contained-at-infty} is satisfied so that $(\phi^\infty,\Iu^\infty)$ contains $(\psi,\Iv)$ at infinity.
\end{proof}

\begin{defi}
Suppose $\beta$ is strongly stable.
Let $\FC$ be a set of cocycle representations from $(A,\alpha)$ to $(\CM(B),\beta)$ that is closed under unitary equivalence with respect to unitaries in $\CM(B)^\beta$.
We say that $\FC$ is \emph{$\sigma$-additive}, if for every sequence $(\phi^{(n)},\Iu^{(n)})_{n\in\IN}$ in $\FC$, the Cuntz sum
\[
\bigoplus_{n\in\IN} ( \phi^{(n)}, \Iu^{(n)} ) : (A,\alpha) \to (\CM(B),\beta)
\]
also belongs to $\FC$.
We say that a cocycle representation $(\theta,\Ix)\in\FC$ is \emph{absorbing in $\FC$}, if $(\theta\oplus\phi,\Ix\oplus\Iu)\asymp (\theta,\Ix)$ for all $(\phi,\Iu)\in\FC$.
\end{defi}

\begin{example}
It is clear from the definition that for any family of $\sigma$-additive sets of cocycle representations in the above sense, their intersection is also $\sigma$-additive.
Keeping this in mind, the following sets of cocycle representations from $(A,\alpha)$ to $(\CM(B),\beta)$ are $\sigma$-additive, as well as any intersection of such sets with each other.
  \begin{itemize}[leftmargin=5mm]
  \item The set of all cocycle representations $(\phi,\Iu): (A,\alpha)\to (\CM(B),\beta)$.
  An absorbing element in this class will simply be called an absorbing cocycle representation.
  \item Assuming $A$ is unital, the set of all cocycle representations $(\phi,\Iu): (A,\alpha)\to (\CM(B),\beta)$ such that $\phi$ is unital.
  \item The set of all cocycle representations $(\phi,\Iu): (A,\alpha)\to (\CM(B),\beta)$ such that $\phi$ is weakly nuclear.
  \item The set of all cocycle representations $(\phi,\Iu): (A,\alpha)\to (\CM(B),\beta)$ such that $\phi$ weakly belongs to a predetermined closed operator convex cone of completely positive maps $A\to B$; cf.\ \cite{KirchbergRordam02, GabeRuiz15}.
  \end{itemize}
\end{example}

\begin{rem}
Let us equip the set of all cocycle representations $(\phi,\Iu): (A,\alpha)\to (\CM(B),\beta)$ with the topology of point-strict convergence in the first variable, and uniform strict convergence over compact sets of $G$ in the second variable.
Let us denote this topology on the set of all cocycle representations by $\tau(\alpha,\beta)$.
As we assume that $B$ is $\sigma$-unital, the strict topology is metrizable on the unit ball of $\CM(B)$.
As $A$ is assumed to be separable and $G$ is second-countable, it follows that the topology $\tau(\alpha,\beta)$ is metrizable.
If it is moreover the case that $B$ is separable, then $\CM(B)$ is strictly separable and the topology $\tau(\alpha,\beta)$ is automatically separable.
\end{rem}

The following existence result for absorbing elements generalizes \cite[Theorems 2.4+2.7]{Thomsen01}, \cite[Theorem 3.14]{GabeRuiz15}, \cite[Theorem 5.7]{Thomsen05} and various similar results scattered throughout the literature.
(At first glance one might think that the cited theorem is partially stronger than the one proved here when it applies, but this is the case only to the extent that the underlying $*$-homomorphism $A\to\CM(B)$ can be chosen to be extendible in \cite{Thomsen05}.)
We note that our proof is partially new, but the new part is based only on a string of very elementary observations (plus some previous results in this section), which yields a drastic simplification even in the classical non-equivariant case that is well-known in the literature.
In particular, compared to the known proofs, it is not necessary to invoke the Kasparov--Stinespring dilation for completely positive maps.

\begin{theorem} \label{thm:abs-rep-existence}
Suppose $\beta$ is strongly stable.
Let $\FC$ be a $\sigma$-additive set of cocycle representations from $(A,\alpha)$ to $(\CM(B),\beta)$.
If $\FC$ is $\tau(\alpha,\beta)$-separable, then there exists an absorbing element in $\FC$.
In particular, if $B$ is separable, then all $\sigma$-additive sets of cocycle representations have an absorbing element.
\end{theorem}
\begin{proof}
As $\FC$ is $\sigma$-additive, it follows from \autoref{cor:coc-rep-absorption} that if there exists some $(\theta,\Ix)\in\FC$ such that $(\phi,\Iu)\wc (\theta,\Ix)$ for every $(\phi,\Iu)\in\FC$, then its infinite repeat $(\theta^\infty,\Ix^\infty)$ is absorbing.
By the assumption that $\FC$ is separable with respect to $\tau(\alpha,\beta)$, we can find a sequence of elements $(\theta^{(n)},\Ix^{(n)})\in\FC$ with the following property.
Given any cocycle representation $(\phi,\Iu)\in\FC$, any compacts sets $K\subseteq G$, $\CF^A\subset A$, $b\in B$, and $\eps\greater 0$, there exists some $n$ such that
\[
\max_{a\in\CF^A} \| (\phi(a)-\theta^{(n)}(a))b \| \leq \eps \quad\text{and}\quad
\max_{g\in K} \| (\Iu_g-\Ix^{(n)}_g)b \| \leq \eps.
\]
But due to \autoref{prop:wc-observations}, this implies $(\phi,\Iu)\wc\bigoplus_{n=1}^\infty (\theta^{(n)},\Ix^{(n)}) =: (\theta,\Ix)$ for every $(\phi,\Iu)\in\FC$.
This completes the proof.
\end{proof}

\begin{cor} \label{lem:standard-form-KK-classes}
Suppose $B$ is separable and $\beta$ is strongly stable.
Let $(\theta,\Iy): (A,\alpha)\to (\CM(B),\beta)$ be an absorbing cocycle representation (which exists by \autoref{thm:abs-rep-existence}).
Then:
\begin{enumerate}[label=\textup{(\roman*)},leftmargin=*] 
\item Every element $x\in \IE^G(\alpha,\beta)/{\sim_h}$ can be expressed as the equivalence class of an $(\alpha,\beta)$-Cuntz pair $\big( (\phi,\Iu), (\theta,\Iy) \big)$ for some absorbing cocycle representation $(\phi,\Iu): (A,\alpha)\to (\CM(B),\beta)$. \label{lem:standard-form-KK-classes:1}
\item Every element $z\in KK^G(\alpha,\beta)$ can be expressed as the equivalence class of an anchored $(\alpha,\beta)$-Cuntz pair $\big( (\phi,\Iu), (\theta,\Iy) \big)$ for some absorbing cocycle representation $(\phi,\Iu): (A,\alpha)\to (\CM(B),\beta)$. \label{lem:standard-form-KK-classes:2}
\end{enumerate}
\end{cor}
\begin{proof}
In light of \autoref{prop:KKG-homotopy}, part \ref{lem:standard-form-KK-classes:2} is a special case of \ref{lem:standard-form-KK-classes:1}, so we shall prove the latter.
Let us write $x=\big[ (\psi,\Iv), (\kappa,\Ix) \big]$ for an arbitrary $(\alpha,\beta)$-Cuntz pair of cocycle representations.
Since $(\theta,\Iy)$ is absorbing, it absorbs $(\kappa,\Ix)$, i.e., we find a norm-continuous path of unitaries $w: [0,\infty)\to\CU (\CM(B))$ satisfying
\[
\lim_{t\to\infty} \Big( \|w_t^*(\kappa\oplus\theta)(a)w_t-\theta(a)\| + \max_{g\in K}\|w_t^*(\Ix_g\oplus\Iy_g)\beta_g(w_t)-\Iy_g\| \Big)=0
\]
for all $a\in A$ and every compact set $K\subseteq G$, and so that the expressions in the norms above belong to $B$ pointwise.
By these properties, we have that the two cocycle representations
\[
(\theta,\Iy) \quad\text{and}\quad \ad(w_0^*)\circ(\kappa\oplus\theta,\Ix\oplus\Iy)
\]
form an $(\alpha,\beta)$-Cuntz pair that is homotopic to a degenerate one.
Hence we observe the following equality of homotopy classes:
\[
\begin{array}{ccl}
\big[ (\psi,\Iv), (\kappa,\Ix) \big] &=& \big[ (\psi\oplus\theta,\Iv\oplus\Iy), (\kappa\oplus\theta,\Ix\oplus\Iy) \big] \\
&=& \big[ \ad(w_0^*)\circ(\psi\oplus\theta,\Iv\oplus\Iy), \ad(w_0^*)\circ(\kappa\oplus\theta,\Ix\oplus\Iy) \big] \\
&=& \big[ \ad(w_0^*)\circ(\psi\oplus\theta,\Iv\oplus\Iy), \ad(w_0^*)\circ(\kappa\oplus\theta,\Ix\oplus\Iy) \big] \\
&&+ \big[ \ad(w_0^*)\circ(\kappa\oplus\theta,\Ix\oplus\Iy), (\theta,\Iy) \big] \\
&=& \big[ \ad(w_0^*)\circ(\psi\oplus\theta,\Iv\oplus\Iy), (\theta,\Iy) \big].
\end{array}
\]
Therefore the cocycle representation $(\phi,\Iu)= \ad(w_0^*)\circ(\psi\oplus\theta,\Iv\oplus\Iy)$ has the desired property.
Since $(\theta,\Iy)$ was absorbing, clearly so is $(\phi,\Iu)$ by construction.
This finishes the proof.
\end{proof}


\section{Criteria for asymptotic unitary equivalence}

In this section we provide a general argument showing that if one compares two cocycle representations forming a Cuntz pair, then strong asymptotic unitary equivalence (see below) is implied by an a priori weaker notion of equivalence, which is in turn implied by operator homotopy.
This will provide the final technical ingredient towards our main result in the next section, and replaces all the arguments related to derivations \cite[Subsection 2.3]{DadarlatEilers01} appearing in the prior proof of the stable uniqueness theorem.

\begin{defi} \label{def:properasym}
Let $(\phi,\Iu), (\psi,\Iv): (A,\alpha) \to (\CM(B),\beta)$ be two cocycle representations.
We say that $(\phi,\Iu)$ and $(\psi,\Iv)$ are \emph{properly asymptotically unitarily equivalent},
if there exists a norm-continuous path $u: [0,\infty)\to\CU(\eins+B)$ such that
\[
\lim_{t\to\infty} \|\psi(a)-u_t\phi(a)u_t^*\|=0
\]
for all $a\in A$ and
\[
\lim_{t\to\infty} \max_{g\in K}\ \| \Iv_g - u_t\Iu_g\beta_g(u_t)^* \| =0
\] 
for all compact sets $K\subseteq G$.
If one may arrange $u_0=\eins$, then we call $(\phi,\Iu)$ and $(\psi,\Iv)$ \emph{strongly asymptotically unitarily equivalent}.
\end{defi}

\begin{lemma} \label{lem:approximate-multipliers-1}
Let $B$ be a $\sigma$-unital \cstar-algebra with an action $\beta: G\curvearrowright B$.
Let $D\subseteq\CM(B)$ be a separable \cstar-subalgebra.
Let $u: [0,\infty)\to\CU(\CM^\beta(B))$ be a norm-continuous path of unitaries with $u_0=\eins$ such that $u|_{[1,\infty)}$ is constant and $[u_t, D]\subseteq B$ for all $t\in [0,1]$.
Suppose that $\max_{0\leq t\leq 1} \|u_t-\eins\|<2$. 
Then for all sequences of $\eps_n>0$, compact subsets $\CF_n\subset B$, $\CG_n\subset D+B$ and $K_n\subseteq G$, where $n\geq 0$, there exists a norm-continuous path of unitaries $v: [0,\infty)\to\CU(\eins+B)$ with $v_0=\eins$ such that
\[
\max_{n\leq t\leq n+1} \Big( \max_{b\in\CF_n} \|(v_t^*u_t-\eins)b\| + \max_{d\in\CG_n} \|[ v_t^*u_t, d]\| + \max_{g\in K_n} \|v_t^*u_t-\beta_g(v_t^*u_t)\| \Big) \leq \eps_n
\]
for all $n\geq 0$.
\end{lemma}
\begin{proof}
Without loss of generality, we may assume that both $\CF_n$ and $\CG_n$ consist of  contractions.
Note that the set of all elements $x\in\CM^\beta(B)$ with $[x,D]\subseteq B$ forms a \cstar-algebra.
Because the distance from the path $u$ to the unit is uniformly bounded away from 2,  we can apply functional calculus and find a norm-continuous path of self-adjoints $a: [0,1]\to \CM^\beta(B)$ with $a_0=0$ and
\[
\max_{0\leq t\leq 1} \|a_t\|<\pi ,\quad [a_t,D]\subseteq B ,\quad \text{and}\quad u_t=\exp(ia_t),\ t\in [0,1].
\]
For convenience, let us define $a_t=a_1$ for $t\geq 1$.

By approximating the exponential function via its partial power series, it is completely standard to show the following fact.
For every $\eps>0$, there exists $\delta>0$ such that the following properties hold for any pair of elements $c,d\in A$ in any Banach algebra $A$:
\begin{itemize}[leftmargin=*]
\item If $\|c\|\leq 2\pi$ and $\|d\|\leq\delta$, then $\|\exp(c+d)-\exp(c)\|\leq\eps$.
\item If $\|c\|,\|d\|\leq\pi$, then $\|[c,d]\| \leq \delta$ implies $\|\exp(c)\exp(d)-\exp(c+d)\|\leq\eps$.
\item If $\|c\|\leq 1$ and $\|d\|\leq 2\pi$, then $\|[c,d]\| \leq \delta$ implies $\|[c,\exp(d)]\|\leq\eps$.
\item If $\|c\|\leq 1$ and $\|d\| \leq 2\pi$, then $\|dc\|\leq\delta$ implies $\|(\exp(d)-\eins)c\|\leq\eps$.
\end{itemize}
For every $n\geq 0$, we apply this fact to choose a constant $\delta=\delta_n$ with $\eps_n/8$ in place of $\eps$.
By \autoref{lem:Kasparov}, it is possible to find an increasing approximate unit $h_n\in B$ with $\max_{g\in K} \|h_n-\beta_g(h_n)\|\to 0$ for every compact set $K\subseteq G$, and moreover satisfying the quasicentrality condition
\[
\lim_{n\to\infty} \max_{0\leq t\leq 1} \| [a_t, h_n] \| = 0 = \lim_{n\to\infty}  \| [d, h_n] \|
\]
for all $d\in D$.
By linear interpolation, we may extend this sequence to an increasing norm-continuous map of positive contractions $h: [0,\infty)\to B$ with the same asymptotic properties.
In particular we observe
\[
\begin{array}{cl}
\multicolumn{2}{l}{ \dst \lim_{s\to\infty} \max_{g\in K} \|(a_t- h_s a_t h_s)-\beta_g(a_t- h_s a_t h_s)\| } \\
=& \dst \lim_{s\to\infty} \max_{g\in K} \| (\eins-h_s^2)\underbrace{(a_t-\beta_g(a_t))}_{\in B}\| \ = \ 0
\end{array}
\]
for every compact set $K\subseteq G$, and uniformly over all $t\in [0,1]$.
Similarly one has for all $d\in D+B$ that
\[
\begin{array}{cl}
\multicolumn{2}{l}{ \dst \lim_{s\to\infty} \max_{g\in K} \| \big[ (a_t- h_s a_t h_s), d \big] \| } \\
=& \dst \lim_{s\to\infty} \max_{g\in K} \| (\eins-h_s^2)\underbrace{[a_t, d]}_{\in B}\| \ = \ 0 
\end{array}
\]
for every compact set $K\subseteq G$ and uniformly over all $t\in [0,1]$.

By reparameterizing $h$ and/or cutting away an initial segment of the interval, if necessary, we may ensure that the following estimates are true for every $n\geq 0$:
\begin{equation} \label{eq:approx-mult-1:1}
\sup_{n\leq t\leq n+1} \max_{b\in\CF_n} \|(a_t-h_t a_t h_t)b\|\leq\delta_n;
\end{equation} 
\begin{equation} \label{eq:approx-mult-1:2}
\sup_{n\leq t\leq n+1}  \|[a_t, h_ta_th_t]\|\leq\delta_n;
\end{equation}
\begin{equation} \label{eq:approx-mult-1:3}
\sup_{n\leq t\leq n+1} \ \max_{g\in K_n} \|(a_t- h_t a_t h_t)-\beta_g(a_t- h_t a_t h_t)\| \leq \delta_n;
\end{equation}
\begin{equation} \label{eq:approx-mult-1:4}
\sup_{n\leq t\leq n+1} \ \max_{d\in \CG_n} \| \big[ (a_t- h_t a_t h_t), d \big] \| \leq \delta_n.
\end{equation}
We define $v_t=\exp(ih_ta_th_t)\in\CU(\eins+B)$ and claim that this satisfies the required properties.
Since $a_0=0$, we have $v_0=\eins$.
Due to the choice of the constant $\delta_n$, we can compute for all $n\geq 0$, $t\in [n,n+1]$ and $b\in\CF_n$ that
\[
\begin{array}{ccl}
v_t^*u_tb & \hspace{4.5mm} = \hspace{4.5mm} & \exp(-ih_ta_th_t)\exp(ia_t)b \\ 
&\stackrel{\eqref{eq:approx-mult-1:2}}{=}_{\makebox[0pt]{\footnotesize\hspace{2mm}$\eps_n/8$}}& \exp(i(a_t-h_ta_th_t))b \ \stackrel{\eqref{eq:approx-mult-1:1}}{=}_{\makebox[0pt]{\footnotesize\hspace{2mm}$\eps_n/8$}} \hspace{5mm} b.
\end{array}
\]
Furthermore we compute for all $n\geq 0$, $t\in [n,n+1]$ and $g\in K_n$ that
\[
\begin{array}{cclcl}
\beta_g(v_t^*u_t) & \stackrel{\eqref{eq:approx-mult-1:2}}{=}_{\makebox[0pt]{\footnotesize\hspace{2mm}$\eps_n/8$}}\hspace{2.5mm} & \beta_g(\exp(i(a_t-h_ta_th_t)))  
& \hspace{4mm}= \hspace{4mm} & \exp(i\beta_g(a_t-h_ta_th_t)) \\
&\stackrel{\eqref{eq:approx-mult-1:3}}{=}_{\makebox[0pt]{\footnotesize\hspace{2mm}$\eps_n/8$}}\hspace{2.5mm}&
\exp(i(a_t-h_ta_th_t)) & \stackrel{\eqref{eq:approx-mult-1:2}}{=}_{\makebox[0pt]{\footnotesize\hspace{2mm}$\eps_n/8$}}\vspace{4mm} & v_t^*u_t.
\end{array}
\]
Lastly we compute for all $n\geq 0$, $t\in [n,n+1]$ and $d\in\CG_n$ that
\[
[v_t^*u_t, d] \stackrel{\eqref{eq:approx-mult-1:2}}{=}_{\makebox[0pt]{\footnotesize\hspace{3.5mm}$2\eps_n/8$}} \hspace{6mm}  \big[ \exp(i(a_t-h_ta_th_t)), d \big]  \ \stackrel{\eqref{eq:approx-mult-1:4}}{=}_{\makebox[0pt]{\footnotesize\hspace{2mm}$\eps_n/8$}} \hspace{5mm} 0.
\]
The claim follows via the triangle inequality.
\end{proof}

\begin{lemma} \label{lem:approximate-multipliers}
Let $B$ be a $\sigma$-unital \cstar-algebra with an action $\beta: G\curvearrowright B$.
Let $D\subseteq\CM(B)$ be a separable \cstar-subalgebra.
Let $U: [0,\infty)\to\CU(\CM^\beta(B))$ be a norm-continuous path of unitaries with $U_0=\eins$ and $[U_t,D]\subseteq B$ for all $t\geq 0$.
Then there exists a norm-continuous path of unitaries $v: [0,\infty)\to\CU(\eins+B)$ with $v_0=\eins$ such that
\[
\lim_{t\to\infty} \|(v^*_tU_t-\eins)b\| = 0 = \lim_{t\to\infty} \| \big[ v_t^*U_t, d \big] \|
\]
for all $b\in B$ and $d\in D$, and
\[
\lim_{t\to\infty}\max_{g\in K} \|v_t^*U_t-\beta_g(v_t^*U_t)\|=0
\]
for every compact set $K\subseteq G$.
\end{lemma}
\begin{proof}
Since $B$ is $\sigma$-unital, it has a strictly positive element $h\in B$, which we fix for the rest of the proof. 
The first of the above limit properties holds if it holds for $h$ in place of $b$.
Using that $D$ is separable, we choose a compact subset $\CG\subset D$ of contractions whose linear span is dense in $D$.
The second of the above limit properties holds if it holds for $d\in\CG$.

By using the fact that $U$ is uniformly continuous on bounded intervals, we are able to find an increasing sequence $0=t_0<t_1<t_2<\dots$ of real numbers with $t_n\to\infty$ such that
\[
\max_{t_n\leq s\leq t_{n+1}} \|U_s-U_{t_n}\|<2,\quad n\geq 0.
\]
After reparametrizing the path, if necessary, we may assume $t_n=n$ for all $n\geq 1$.
Let us define for every $n\geq 0$ a norm-continuous path of unitaries $U^{(n)}: [0,\infty)\to\CU(\CM^\beta(B))$ via
\[
U^{(n)}_t = \begin{cases} U_{n+1} U_{n}^* &,\quad t> n+1 \\
U_t U_{n}^* &,\quad n \leq t \leq n+1 \\
\eins &,\quad t < n.
 \end{cases}
\]
By construction we thus have for all $n\geq 0$ and $t \leq n+1$ the equality
\[
U_t = U^{(n)}_t U^{(n-1)}_t \cdots U^{(0)}_t.
\]
Let $K_n\subseteq G$ be an increasing sequence of compact sets with $G=\bigcup_{n\geq 0} \operatorname{int}(K_n)$.
We consider the compact sets $\CF_{n}^k\subset B$ for natural numbers $n\leq k$ via
\[
\begin{array}{cl}
\CF_n^k = & \set{ U^{(n-1)}_t \cdots U^{(0)}_t h \mid t\leq k } \\
& \cup \set{ U^{(n-1)}_t\cdots U^{(j+1)}_t ( U^{(j)}_t\beta_g(U^{(j)}_t)^*-\eins) \mid 0\leq j<n,\ g\in K_k,\ t\leq k } \\
& \cup \set{ U^{(n-1)}_t\cdots U^{(j+1)}_t ( \beta_g(U^{(j)}_t)U^{(j)*}_t-\eins) \mid 0\leq j<n,\ g\in K_k,\ t\leq k }.
\end{array}
\]
Here we implicitly follow the convention that the product $U^{(n-1)}_t\cdots U^{(j+1)}_t$ is understood as the unit when $j=n-1$, as the upper indices are supposed to descend from left to right and we end up with the empty product.
Note that since the unitary path $U$ takes values in $\CM^\beta(B)$, the elements $U^{(j)}_t\beta_g(U^{(j)}_t)^*$ and $\beta_g(U^{(j)}_t)U^{(j)*}_t$ appearing in this set belong to $\eins+B$, hence indeed $\CF_n^k\subset B$.
We also consider the compact subsets $\CG_n^k\subset D+B$ for $n\leq k$ via
\[
\CG_n^k = \set{ U^{(n-1)}_t \cdots U^{(0)}_t d U^{(0)*}_t \cdots U^{(n-1)*}_t \mid d\in\CG,\ t\leq k+1 }.
\]
We apply \autoref{lem:approximate-multipliers-1} for every $n\geq 0$ and choose a unitary path $v^{(n)}: [0,\infty)\to\CU(\eins+B)$ such that $v^{(n)}|_{[0,n]}=\eins$ and for every natural number $k\geq n$, we have (note that for $k<n$, the norms appearing here are zero)
\begin{equation} \label{eq:approx-mult:1}
\max_{k\leq t\leq k+1} \max_{b\in \CF_n^k} \| (v^{(n)*}_t U^{(n)}_t-\eins) b \| \leq 2^{-(n+k)};
\end{equation}
\begin{equation} \label{eq:approx-mult:2}
\max_{k\leq t\leq k+1} \ \max_{g\in K_k} \| (v^{(n)*}_t U^{(n)}_t)-\beta_g(v^{(n)*}_t U^{(n)}_t) \| \leq 2^{-(n+k)};
\end{equation}
\begin{equation} \label{eq:approx-mult:3}
\max_{k\leq t\leq k+1} \ \max_{d\in \CG_n^k} \| \big[ v^{(n)*}_t U^{(n)}_t , d  \big] \| \leq 2^{-(n+k)};
\end{equation}
Finally, we define $v: [0,\infty)\to\CU(\eins+B)$ via
\[
v_t = v^{(n)}_t v^{(n-1)}_t \cdots v^{(0)}_t,\quad n\leq t\leq n+1.
\]
We claim that this map satisfies the desired properties.

We first estimate for all $t\in [k,k+1]$ that
\[
\begin{array}{cl}
\multicolumn{2}{l}{ \| (v_t^*U_t-\eins) h\| } \\
=& \| (v^{(0)*}_t v^{(1)*}_t \cdots v^{(k)*}_t U^{(k)}_t U^{(k-1)}_t \cdots U^{(0)}_t-\eins)h\| \\
\stackrel{\eqref{eq:approx-mult:1}}{\leq}& 2^{-2k}+\| (v^{(0)*}_t v^{(1)*}_t \cdots v^{(k-1)*}_t U^{(k-1)}_t U^{(k-2)}_t \cdots U^{(0)}_t-\eins)h\| \\
\stackrel{\eqref{eq:approx-mult:1}}{\leq}& 2^{-2k}+2^{-(2k-1)}+\| (v^{(0)*}_t v^{(1)*}_t \cdots v^{(k-2)*}_t U^{(k-2)}_t U^{(k-3)}_t \cdots U^{(0)}_t-\eins)h\| \\
\leq& \dots \ \leq \ \dst \sum_{\ell=k}^{2k} 2^{-\ell} \ \leq \ 2^{1-k}.
\end{array}
\]
Similarly we have for all $t\in [k,k+1]$ and $d\in\CG$ that
\[
\begin{array}{cl}
\multicolumn{2}{l}{ \| \big[ v_t^*U_t, d \big] \| } \\
=& \| d - v_t^*U_t d U_t^* v_t \| \\
=& \| d - v^{(0)*}_t  \cdots v^{(k)*}_t U^{(k)}_t \underbrace{ \cdots U^{(0)}_t d U^{(0)*}_t \cdots  }_{\in\CG_k^k} U^{(k)*}_t v^{(k)}_t\cdots v^{(0)}_t \| \\
\stackrel{\eqref{eq:approx-mult:3}}{\leq}& 2^{-2k} \\
&+ \| d - v^{(0)*}_t  \cdots v^{(k-1)*}_t U^{(k-1)}_t \underbrace{ \cdots U^{(0)}_t d U^{(0)*}_t  \cdots }_{\in\CG_{k-1}^{k}} U^{(k-1)*}_t v^{(k-1)}_t\cdots v^{(0)}_t \| \\
\stackrel{\eqref{eq:approx-mult:3}}{\leq}& 2^{-2k}+2^{-(2k-1)} \\
& + \| d - v^{(0)*}_t  \cdots v^{(k-2)*}_t U^{(k-2)}_t \cdots U^{(0)}_t d U^{(0)*}_t \cdots U^{(k-2)*}_t v^{(k-2)}_t\cdots v^{(0)}_t \| \\
\leq& \dots  \ \leq \ \dst \sum_{\ell=k}^{2k} 2^{-\ell} \ \leq \ 2^{1-k}.
\end{array}
\]
Now we want to prove for every $g\in K_k$ the inequality 
\[
\|v_t^*U_t-\beta_g(v_t^*U_t)\|\leq 3\cdot 2^{1-k},\quad t\in [k,k+1].
\]
This would clearly finish the proof.
For convenience, let us set 
\[
X_t^{(j,k)} = v^{(k-j)*}_t  \cdots v^{(k)*}_t U^{(k)}_t  \cdots U^{(k-j)}_t \quad\text{for } 0\leq j\leq k \text{ and } t\in [0,\infty).
\]
Note that $X^{(k,k)}_t=v_t^*U_t$ when $t\in [k,k+1]$.
We will show inductively that for all natural numbers $j\leq k$, all $t\in [k,k+1]$ and $g\in K_k$ we have the inequality
\begin{equation} \label{eq:approx-mult:4}
\| X_t^{(j,k)} - \beta_g(X_t^{(j,k)}) \| \leq 3 \sum_{n=2k-j}^{2k} 2^{-n}.
\end{equation}
The case $j=k$ then evidently yields the desired inequality and would complete the proof.
For all the computations below, let us fix some $k\geq0$, $g\in K_k$ and $t\in [k,k+1]$.
We first observe that when $j<k$, we have 
\[
\begin{array}{cl}
\multicolumn{2}{l}{ \| (X^{(j,k)}_t-\eins) \big( U^{(k-j-1)}_t \beta_g(U^{(k-j-1)}_t)^*-\eins\big) \| } \\
=& \| (v^{(k-j)*}_t \cdots v^{(k)*}_t U^{(k)}_t \cdots U^{(k-j)}_t-\eins) \big( U^{(k-j-1)}_t \beta_g(U^{(k-j-1)}_t)^*-\eins\big) \| \\
\stackrel{\eqref{eq:approx-mult:1}}{\leq}& 2^{-2k} \\
&+ \| (v^{(k-j)*}_t \cdots v^{(k-1)*}_t U^{(k-1)}_t \cdots U^{(k-j)}_t-\eins) \big( U^{(k-j-1)}_t \beta_g(U^{(k-j-1)}_t)^*-\eins\big) \| \\
\stackrel{\eqref{eq:approx-mult:1}}{\leq}& 2^{-2k} + 2^{-(2k-1)} \\
& + \| (v^{(k-j)*}_t \cdots v^{(k-2)*}_t U^{(k-2)}_t \cdots U^{(k-j)}_t-\eins) \big( U^{(k-j-1)}_t \beta_g(U^{(k-j-1)}_t)^*-\eins\big) \| \\
\leq& \dots \\
\leq& \dst \sum_{n=2k-j}^{2k} 2^{-n} \ \leq \ 2^{-(2k-j-1)}.
\end{array}
\]
In a completely analogous fashion we may also estimate
\[
\begin{array}{cl}
\multicolumn{2}{l}{ \| \big( U^{(k-j-1)}_t \beta_g(U^{(k-j-1)}_t)^*-\eins\big) (X^{(j,k)}_t-\eins) \| }\\
= & \| (X^{(j,k)*}_t-\eins) \big(  \beta_g(U^{(k-j-1)}_t)U^{(k-j-1)*}_t-\eins\big) \| \\
= & \| (\eins-X^{(j,k)}_t) \big(  \beta_g(U^{(k-j-1)}_t)U^{(k-j-1)*}_t-\eins\big) \| \ \leq \ 2^{-(2k-j-1)}. 
\end{array}
\]
These two estimates lead to the inequality
\[
\begin{array}{cl}
\multicolumn{2}{l}{ \| X_t^{(j,k)} U^{(k-j-1)}_t \beta_g(U^{(k-j-1)}_t)^* - U^{(k-j-1)}_t \beta_g(U^{(k-j-1)}_t)^* X_t^{(j,k)}  \| }\\
=& \| (X^{(j,k)}_t-\eins) \big( U^{(k-j-1)}_t \beta_g(U^{(k-j-1)}_t)^*-\eins\big) \\
&- \big( U^{(k-j-1)}_t \beta_g(U^{(k-j-1)}_t)^*-\eins\big) (X^{(j,k)}_t-\eins)\| \\
\leq&  2\cdot 2^{-(2k-j-1)}.
\end{array}
\]
Let us now prove \eqref{eq:approx-mult:4} by induction over $j\leq k$.
We begin with $j=0$.
Here we directly have $\| (v^{(k)*}_t U^{(k)}_t)-\beta_g(v^{(k)*}_t U^{(k)}_t) \| \leq 2^{-2k}$ by \eqref{eq:approx-mult:2}, so there is nothing to show.
Let us proceed with the induction step, and assume that the claimed estimate \eqref{eq:approx-mult:4} holds for some natural number $j<k$.
In order to get it for $j+1$, we use the induction hypothesis and the other estimate from above and compute
\[
\begin{array}{cl}
\multicolumn{2}{l}{ \| X_t^{(j+1,k)} - \beta_g(X_t^{(j+1,k)}) \| } \\
=& \| v^{(k-j-1)*}_t X_t^{(j,k)}  U^{(k-j-1)}_t - \beta_g\big( v^{(k-j-1)*}_t X_t^{(j,k)} U^{(k-j-1)}_t \big) \| \\
\leq& \dst 3 \sum_{n=2k-j}^{2k} 2^{-n} \\
&+ \| v^{(k-j-1)*}_t X_t^{(j,k)} U^{(k-j-1)}_t - \beta_g( v^{(k-j-1)*}_t) X_t^{(j,k)} \beta_g(U^{(k-j-1)}_t) \| \\
=& \dst 3\sum_{n=2k-j}^{2k} 2^{-n} \\
&+ \|  X_t^{(j,k)} U^{(k-j-1)}_t \beta_g(U^{(k-j-1)}_t)^* - v^{(k-j-1)}\beta_g( v^{(k-j-1)*}_t) X_t^{(j,k)}  \| \\
\stackrel{\eqref{eq:approx-mult:2}}{\leq}& \dst 2^{-(2k-j-1)}+3\sum_{n=2k-j}^{2k} 2^{-n}  \\
& + \|  X_t^{(j,k)} U^{(k-j-1)}_t \beta_g(U^{(k-j-1)}_t)^* - U^{(k-j-1)}_t \beta_g(U^{(k-j-1)}_t)^* X_t^{(j,k)}  \| \\
\leq& \dst 3 \sum_{n=2k-j-1}^{2k} 2^{-n} .
\end{array}
\]
As pointed out above, the proof is complete by considering $j=k$.
\end{proof}

The following two corollaries represent the main technical achievement of this section.

\begin{cor} \label{cor:asymptotic-unitary-equivalence}
Let $A$ be a separable \cstar-algebra and $B$ a $\sigma$-unital \cstar-algebra.
Let $\alpha: G\curvearrowright A$ and $\beta: G\curvearrowright B$ be two actions, and let
\[
(\phi,\Iu), (\psi,\Iv): (A,\alpha)\to (\CM(B),\beta)
\]
be two cocycle representations.
Suppose that there exists a norm-continuous path of unitaries $U: [0,\infty)\to\CU(D_{(\phi,\Iu)})$ with $U_0=\eins$ and
\[
\lim_{t\to\infty} \Big( \|\psi(a)-U_t\phi(a)U_t^*\| + \max_{g\in K} \|\Iv_g-U_t\Iu_g\beta_g(U_t)^*\| \Big) = 0
\]
for all $a\in A$ and every compact set $K\subseteq G$.
Then $(\phi,\Iu)$ and $(\psi,\Iv)$ are strongly asymptotically unitarily equivalent.
\end{cor}
\begin{proof}
By definition of the \cstar-algebra $D_{(\phi,\Iu)}\subseteq\CM(B)$, we have $U_t\in\CM^{\beta^\Iu}(B)$ and $[U_t,\phi(A)]\subseteq B$ for all $t\geq 0$.
Using \autoref{lem:approximate-multipliers} for $\beta^\Iu$ in place of $\beta$ and $\phi(A)$ in place of $D$, we may find a norm-continuous path of unitaries $v: [0,\infty)\to\CU(\eins+B)$ with $v_0=\eins$ such that
\[
 \lim_{t\to\infty} \| [v_t^*U_t,\phi(a)] \| = 0
\]
for all $a\in A$, and
\[
\lim_{t\to\infty}\max_{g\in K} \| v_t^*U_t - \beta_g^\Iu(v_t^*U_t)\|=0
\]
for every compact set $K\subseteq G$.
This implies for all $a\in A$ that
\[
\psi(a)=\lim_{t\to\infty} U_t\phi(a)U_t^* = \lim_{t\to\infty} v_t(v_t^*U_t)\phi(a)(v_t^*U_t)^*v_t^* = \lim_{t\to\infty} v_t \phi(a) v_t^*.
\]
Similarly we observe for all $g\in G$ that
\[
\begin{array}{cclcl}
\Iv_g &=& \dst \lim_{t\to\infty} U_t\Iu_g\beta_g(U_t)^* 
&=& \dst \lim_{t\to\infty} U_t\beta_g^\Iu(U_t)^*\Iu_g \\
&=& \dst \lim_{t\to\infty} v_t(v_t^*U_t)\beta_g^\Iu(v_t^*U_t)^*\beta_g^\Iu(v_t)^* \Iu_g 
&=& \dst \lim_{t\to\infty} v_t\Iu_g\beta_g(v_t)^*,
\end{array}
\]
and this convergence is uniform over compact subsets of $G$.
This shows that $v$ is a path of unitaries witnessing the claim that $(\phi,\Iu)$ and $(\psi,\Iv)$ are indeed strongly asymptotically unitarily equivalent.
\end{proof}

\begin{cor} \label{cor:op-hom-implies-saue}
Let $A$ be a separable \cstar-algebra and $B$ a $\sigma$-unital \cstar-algebra.
Let $\alpha: G\curvearrowright A$ and $\beta: G\curvearrowright B$ be two actions, and let
\[
(\phi,\Iu), (\psi,\Iv): (A,\alpha)\to (\CM(B),\beta)
\]
be two cocycle representations.
If $(\phi,\Iu)$ and $(\psi,\Iv)$ are operator homotopic, then they are strongly asymptotically unitarily equivalent.
\end{cor}


\section{The dynamical stable uniqueness theorem}

\begin{defi}
Let $(\phi,\Iu), (\psi,\Iv): (A,\alpha) \to (\CM(B),\beta)$ be two cocycle representations.
Suppose that there exists a unital inclusion $\CO_2\subseteq\CM(B)^\beta$.
Then we call $(\phi,\Iu)$ and $(\psi,\Iv)$ \emph{stably properly (resp.\ strongly) asymptotically unitarily equivalent}, if there exists a cocycle representation $(\kappa,\Ix): (A,\alpha)\to(\CM(B),\beta)$ such that $(\phi,\Iu)\oplus(\kappa,\Ix)$ is properly (resp.\ strongly) asymptotically unitarily equivalent to $(\psi,\Iv)\oplus(\kappa,\Ix)$.
\end{defi}

\begin{lemma} \label{lem:pasu}
Suppose $\beta$ is strongly stable.
Let $(\phi,\Iu), (\psi,\Iv): (A,\alpha) \to (\CM(B),\beta)$ be two cocycle representations.
If $(\phi,\Iu)$ and $(\psi,\Iv)$ are properly asymptotically unitarily equivalent, then they form an $(\alpha,\beta)$-Cuntz pair $\big( (\phi,\Iu), (\psi,\Iv) \big)$ that is homotopic to $\big((0, \eins), (0, \eins)\big)$.
\end{lemma}
\begin{proof}
Let $\set{u_t}_{t\geq 1}\subset\CU(\eins+B)$ be a continuous unitary path witnessing the assumption.
Under the quotient map $\CM(B)\to\CQ(B)$, one trivially has $\bar{u}_t= \eins$ for all $t\geq 1$, so the relation $\lim_{t\to\infty} \|\psi(a)-u_t\phi(a)u_t^*\|=0$ implies $\bar{\psi}(a)=\bar{\phi}(a)$ for all $a\in A$.
In other words, one has $\psi(a)-\phi(a)\in B$ for all $a\in A$.
By repeating this argument for the cocycles, we likewise see that $\Iv_g-\Iu_g\in B$ for all $g\in G$.
Hence $(\phi,\Iu)$ and $(\psi,\Iv)$ indeed form a Cuntz pair.

If we set
\[
(\phi^{(t)},\Iu^{(t)}) = \begin{cases} \ad(u_{1/t})\circ(\phi,\Iu) &,\quad t\in (0,1] \\
(\psi,\Iv) &,\quad t=0, \end{cases}
\]
then $[0,1]\ni t\mapsto \big( (\phi^{(t)},\Iu^{(t)}), (\psi,\Iv) \big)$ defines a homotopy between a degenerate Cuntz pair and the pair $\big(\ad(u_1)\circ(\phi,\Iu), (\psi,\Iv) \big)$.
Since $u_1\in\CU(\eins+B)$, the latter Cuntz pair is homotopic to $\big( (\phi,\Iu), (\psi,\Iv) \big)$ by \autoref{prop:Cuntz-pair-small-perturbation}, which with \autoref{lem:equivalence-simplified} finishes the proof.
\end{proof}

The following generalizes \cite[Lemma 3.4]{DadarlatEilers01}.

\begin{lemma} \label{lem:reduction-to-absorbing-case}
Suppose that there exists a unital inclusion $\CO_2\subseteq\CM(B)^\beta$.
Let four cocycle representations
\[
(\phi,\Iu), (\psi,\Iv), (\rho,\Iw), (\theta,\Ix): (A,\alpha) \to (\CM(B),\beta)
\]
be given.
If $(\phi,\Iu)\oplus(\rho,\Iw)$ is strongly asymptotically unitarily equivalent to $(\psi,\Iv)\oplus(\rho,\Iw)$ and $(\rho,\Iw)\asymp (\theta,\Ix)$, then it follows that $(\phi,\Iu)\oplus(\theta,\Ix)$ is strongly asymptotically unitarily equivalent to $(\psi,\Iv)\oplus(\theta,\Ix)$.
\end{lemma}
\begin{proof}
Let $\set{u_t}_{t\geq 1}\subset\CU(\eins+B)$ be a norm-continuous path with $u_0=\eins$ witnessing that $(\phi\oplus\rho,\Iu\oplus\Iw)$ and $(\psi\oplus\rho,\Iv\oplus\Iw)$ are strongly asymptotically unitarily equivalent.
Let $\set{v_t}_{t\geq 1}\subset\CU(\CM(B))$ be a norm-continuous path witnessing $(\rho,\Iw)\asymp (\theta,\Ix)$ in the sense of \autoref{def:various-equivalences}.
Then evidently $w_t=(\eins\oplus v_t)u_t(\eins\oplus v_t^*)$ defines a norm-continuous unitary path in $\CU(\eins+B)$ with $w_0=\eins$, and it satisfies for all $a\in A$ that
\[
\begin{array}{cl}
\multicolumn{2}{l}{ \| (\psi\oplus\theta)(a) - w_t(\phi\oplus\theta)(a)w_t^* \| }\\
=& \| (\psi\oplus\theta)(a) - (\eins\oplus v_t)u_t(\eins\oplus v_t^*)(\phi\oplus\theta)(a)(\eins\oplus v_t)u_t^*(\eins\oplus v_t^*) \| \\
\leq& \| v_t^* \theta(a) v_t - \rho(a)\| + \| (\psi\oplus\theta)(a) - (\eins\oplus v_t)u_t(\phi\oplus\rho)(a)u_t^*(\eins\oplus v_t^*) \| \\
\leq& \| \theta(a)  - v_t\rho(a)v_t^*\| + \| (\psi\oplus\rho)(a)-u_t(\phi\oplus\rho)(a)u_t^*\| \\
& +\| (\psi\oplus\theta)(a) - (\eins\oplus v_t)(\psi\oplus\rho)(a)(\eins\oplus v_t^*) \| \\
\leq& 2\| \theta(a) - v_t\rho(a)v_t^*\| + \| (\psi\oplus\rho)(a)-u_t(\phi\oplus\rho)(a)u_t^*\| \ \to \ 0.
\end{array}
\]
Likewise we see for every compact set $K\subseteq G$ that
\[
\begin{array}{cl}
\multicolumn{2}{l}{ \dst\max_{g\in K} \| (\Iv_g\oplus\Ix_g) - w_t(\Iu_g\oplus\Ix_g)\beta_g(w_t)^* \| }\\
=& \dst\max_{g\in K} \| (\Iv_g\oplus\Ix_g) - (\eins\oplus v_t)u_t (\Iu_g\oplus v_t^*\Ix_g\beta_g(v_t)) \beta_g(u_t)^*(\eins\oplus \beta_g(v_t)^*) \| \\
\leq& \dst\max_{g\in K} \Big( \| v_t^*\Ix_g\beta_g(v_t)-\Iw_g\| \\
& \dst + \| (\Iv_g\oplus\Ix_g) - (\eins\oplus v_t)u_t (\Iu_g\oplus\Iw_g) \beta_g(u_t)^*(\eins\oplus \beta_g(v_t)^*) \| \Big) \\
\leq& \dst\max_{g\in K} \Big( \|\Ix_g-v_t\Iw_g\beta_g(v_t)^*\| + \|(\Iv_g\oplus\Iw_g) - u_t (\Iu_g\oplus\Iw_g) \beta_g(u_t)^* \| \\
& \dst+ \| (\Iv_g\oplus\Ix_g) - (\eins\oplus v_t) (\Iv_g\oplus\Iw_g) (\eins \oplus \beta_g(v_t)^*) \| \Big) \\
=& \dst\max_{g\in K} \Big( 2\|\Ix_g-v_t\Iw_g\beta_g(v_t)^*\| + \|(\Iv_g\oplus\Iw_g) - u_t (\Iu_g\oplus\Iw_g) \beta_g(u_t)^* \|\Big)  \ \to \ 0.
\end{array}
\]
This shows our claim.
\end{proof}

We finally have everything ready to state and prove our main result:

\begin{theorem} \label{thm:main}
Let $A$ be a separable and $B$ a $\sigma$-unital \cstar-algebra, and let $G$ be a second-countable locally compact group.
Let $\alpha: G\curvearrowright A$ and $\beta: G\curvearrowright B$ be two continuous actions.
Let
\[
(\phi,\Iu), (\psi,\Iv): (A,\alpha)\to (\CM(B\otimes\CK),\beta\otimes\id_\CK)
\]
be two cocycle representations that form an anchored $(\alpha,\beta)$-Cuntz pair.
Then the following are equivalent:
\begin{enumerate}[label=\textup{(\roman*)},leftmargin=*]
\item $\big[ (\phi,\Iu), (\psi,\Iv) \big] = 0$ in $KK^G(\alpha,\beta)$. \label{thm:main:1}
\item $(\phi,\Iu)$ and $(\psi,\Iv)$ are stably operator homotopic. \label{thm:main:2}
\item $(\phi,\Iu)$ and $(\psi,\Iv)$ are stably strongly asymptotically unitarily equivalent.\label{thm:main:3}
\item $(\phi,\Iu)$ and $(\psi,\Iv)$ are stably properly asymptotically unitarily equivalent.\label{thm:main:4}
\end{enumerate}
If $B$ is separable, then these statements are further equivalent to
\begin{enumerate}[label=\textup{(\roman*)},leftmargin=*,resume]
\item For every absorbing cocycle representation 
\[
(\theta,\Iy): (A,\alpha)\to(\CM(B\otimes\CK),\beta\otimes\id_\CK),
\] 
one has that $(\phi,\Iu)\oplus(\theta,\Iy)$ is strongly asymptotically unitarily equivalent to $(\psi,\Iv)\oplus(\theta,\Iy)$. \label{thm:main:5}
\end{enumerate}
\end{theorem}
\begin{proof}
The implication \ref{thm:main:1}$\Rightarrow$\ref{thm:main:2} is \autoref{thm:homotopy-implies-operator-homotopy}.
The implication \ref{thm:main:2}$\Rightarrow$\ref{thm:main:3} is clearly a consequence of \autoref{cor:op-hom-implies-saue}.
The implication \ref{thm:main:3}$\Rightarrow$\ref{thm:main:4} is trivial, and \ref{thm:main:4}$\Rightarrow$\ref{thm:main:1} follows from \autoref{lem:pasu}.

Now let us also assume that $B$ is separable.
By \autoref{thm:abs-rep-existence}, there exists an absorbing cocycle representation $(\theta,\Iy)$ as in the statement, one of which we shall now choose.
Therefore clearly \ref{thm:main:5}$\Rightarrow$\ref{thm:main:3}.
So it suffices to show \ref{thm:main:3}$\Rightarrow$\ref{thm:main:5}.
Assuming that $(\phi,\Iu)$ and $(\psi,\Iv)$ are stably strongly asymptotically unitarily equivalent, let $(\kappa,\Ix): (A,\alpha)\to (\CM(B\otimes\CK),\beta\otimes\id_\CK)$ be any cocycle representation such that $(\phi,\Iu)\oplus(\kappa,\Ix)$ is strongly asymptotically unitarily equivalent to $(\psi,\Iv)\oplus(\kappa,\Ix)$.
Then $(\phi,\Iu)\oplus(\kappa,\Ix)\oplus(\theta,\Iy)$ is also strongly asymptotically unitarily equivalent to $(\psi,\Iv)\oplus(\kappa,\Ix)\oplus(\theta,\Iy)$.
Since $(\theta,\Iy)$ is absorbing, we have $(\kappa,\Ix)\oplus(\theta,\Iy) \asymp (\theta,\Iy)$.
Hence the claim follows directly from \autoref{lem:reduction-to-absorbing-case}.
The proof is complete.
\end{proof}


\end{document}